\documentclass[letterpaper,11pt,leqno]{article}

\usepackage[titletoc,toc,title]{appendix}	%To have an appendix at the end

\usepackage{fullpage}		% MAKES IT HAVE SMALLER MARGINS
\usepackage{amsmath, amsthm}
\usepackage{eucal}

\usepackage{amssymb}
\usepackage{latexsym}
\usepackage{graphicx}
\usepackage{amsfonts}

\usepackage[dvipsnames]{xcolor}
\usepackage{colortbl}	%Allows you to put color in the tables rows, columns, or cells.

\usepackage{url}

\usepackage{float}		%Allows you to use \begin{figure}[H] command

\usepackage[colorlinks=true, citecolor=blue, linktocpage]{hyperref}%This will make hyperlinks to labelled parts!!!

\usepackage{mdframed}		%Allows you to put text and math in a framed box

\usepackage{empheq}
\usepackage[most]{tcolorbox}

\newtcbox{\mymath}[1][]{%
    nobeforeafter, math upper, tcbox raise base,
    enhanced, colframe=blue!30!black,
    colback=blue!30, boxrule=1pt,
    #1}

\usepackage{subfigure}	%Use this to have figures placed side by side!!

\usepackage{soul}		%Allows you to highlight text using \hl{Blah}

\usepackage[autolanguage]{numprint}  %For proper spacing for comma as a decimal separtor (e.g., 1,328). For example $\numprint{119 740 619}$ will put commas in appropriate places and adjust spacing accordingly.

\usepackage{bm}		%Allows better math bold face than $\mathbf$. To use just type $\bm{math stuff here}$.

\usepackage{centernot}  % to make the ndivides command

\usepackage[toc]{appendix}	%To have an appendix at the end

\usepackage{chngcntr}		%Allows you to change the counter for figues in the appendix to take into account the section. For example, Figure A.2 is 2nd figure in first appendix. After \appendix put \counterwithin{figure}{section}

\usepackage{mathrsfs}		%Allows you to use mathscript font $\mathscr{Blah}$

\usepackage{enumerate}		%Allows you to use \begin{enumerate}[(i)] which denotes each bullet point as (i), (ii), etc.

\usepackage{xfrac}

\usepackage{wrapfig} %Allows you to wrap text around a figure

\usepackage{enumitem} %We use this solely for purpose of being allowed to use the circled numbers in enumerate environment

\allowdisplaybreaks %This allows long align environments to continue from bottom of one page to top of next page

\definecolor{shadecolor}{gray}{0.90}

\DeclareMathOperator{\newF}{\mathcal{F}}
\newcommand{\divides}{\,\vert\,}
\newcommand{\ndivides}{\centernot\vert}
\newcommand{\FibSeq}{\left(F_n\right)_{n=0}^\infty}
\newcommand{\LucSeq}{\left(L_n\right)_{n=0}^\infty}
\newcommand{\Fmn}{\left(\newF_{m,n}\right)_{n=0}^\infty}
\newcommand{\Ften}{\left(\newF_{10,n}\right)_{n=0}^\infty}
\newcommand{\Fmnsub}{\left(\newF_{m,k+rj}\right)_{j=0}^\infty}
\newcommand{\Ftensub}{\left(\newF_{10,k+rj}\right)_{j=0}^\infty}
\newcommand{\Ftensubfill}[2]{\left(\newF_{10, #1 + #2 j}\right)_{j=0}^\infty}
\newcommand{\FtensubZeroK}{\left(\newF_{10,rj}\right)_{j=0}^\infty}

\newcommand{\F}[1]{\newF_{10, #1}}

\newcommand{\Ftenfill}[2]{\newF_{10, #1 + #2 j}}
\newcommand{\FtensubNkrFwd}{\left(\newF_{10,n}\right)_{n=N_{k,r}}^\infty}
\newcommand{\FtensubNkrRev}{\left(\newF_{10,n}\right)_{n=N_{k,r}}^{-\infty}}
\newcommand{\FtensubCompletePeriod}{\left(\newF_{10,k+rj}\right)_{j=0}^{59}}

\newcommand{\pmodd}[1]{\!\!\!\pmod{#1}}	%This gets rid of that ANNOYING large space before the (mod n) when you use \pmod{n} in displaymath mode (e.g., $$ environment or align environment)

\newcommand{\red}[1]{\textcolor{red}{#1}}

\newcommand{\redbf}[1]{\textcolor{red}{\textbf{#1}}}
\newcommand{\bluebf}[1]{\textcolor{blue}{\textbf{#1}}}

\DeclareMathOperator{\lcm}{lcm}
\DeclareMathOperator{\ind}{ind}

\newcommand*\circled[1]{\tikz[baseline=(char.base)]{
            \node[shape=circle,draw,inner sep=2pt] (char) {#1};}} %makes nice-looking circle around numbers for example, and can be used in text mode.

\newtheorem{theorem}{Theorem}[section] %[section] here insures the start of each section has sets the theorem counter back to 1. 
\newtheorem{lemma}[theorem]{Lemma}
\newtheorem{corollary}[theorem]{Corollary}

\newtheorem{proposition}[theorem]{Proposition}

\newtheorem{formula}[theorem]{Formula}
\newtheorem{algorithm}[theorem]{Algorithm}

\theoremstyle{definition}
\newtheorem{example}[theorem]{Example}

\theoremstyle{definition}
\newtheorem{definition}[theorem]{Definition}

\theoremstyle{plain}

\theoremstyle{definition}
\newtheorem{question}[theorem]{Question}

\theoremstyle{definition}
\newtheorem{remark}[theorem]{Remark}

\theoremstyle{definition}

\numberwithin{figure}{section}
\numberwithin{theorem}{section}
\numberwithin{table}{section}

\newcommand{\nocontentsline}[3]{}
\newcommand{\tocless}[2]{\bgroup\let\addcontentsline=\nocontentsline#1{#2}\egroup}

\makeatletter
\newcommand\ackname{Acknowledgments}
\if@titlepage
  \newenvironment{acknowledgments}{%
      \titlepage
      \null\vfil
      \@beginparpenalty\@lowpenalty
      \begin{center}%
        \bfseries \ackname
        \@endparpenalty\@M
      \end{center}}%
     {\par\vfil\null\endtitlepage}
\else
  
\fi
\makeatother

\newcommand{\modcircleWithBlack}[4]{
\begin{tikzpicture}[scale=#4, every node/.style={scale=#4}]
\draw[blue, thick] (0,0) circle (2in);
\foreach \x in {6,12,...,360} {
\draw (\x: 2in) -- (\x: 2.1in);}
\foreach \x in {0,90,180,270} {
\draw[red, ultra thick] (\x: 2in) -- (\x: 2.1in);
\draw[red, ultra thick, ->] (\x: 2.3in) -- (\x: 2.5in);
}
\foreach \x in {30,60,120,150,210,240,300,330} {
\draw[blue, thick] (\x: 2in) -- (\x: 2.1in);
\draw[blue, thick, ->] (\x: 2.3in) -- (\x: 2.4in);
}
\draw[dotted] (0:2in) -- (180:2in);
\draw[dotted] (90:2in) -- (270:2in);
\foreach \n/\t in {0/\redbf{0},6/7,12/3,18/4,24/9,30/\bluebf{5},36/4,42/1,48/3,54/8,60/\bluebf{5},66/3,72/2,78/1,84/1,90/\redbf{0},96/1,102/9,108/2,114/7,120/\bluebf{5},126/2,132/3,138/9,144/4,150/\bluebf{5},156/9,162/6,168/3,174/3,180/\redbf{0},186/3,192/7,198/6,204/1,210/\bluebf{5},216/6,222/9,228/7,234/2,240/\bluebf{5},246/7,252/8,258/9,264/9,270/\redbf{0},276/9,282/1,288/8,294/3,300/\bluebf{5},306/8,312/7,318/1,324/6,330/\bluebf{5},336/1,342/4,348/7,354/7} {
\draw (\n: 2.2in) node{\t};}

\def\r{-#1*6}
\foreach \k in {0,...,#2} {
\foreach \n in {0,...,#3} {
\draw[red, ultra thin] (\n*\r-6*\k+90: 2in) -- (\n*\r+\r-6*\k+90: 2in);}
}

\def\r{-#1*6}
\foreach \k in {0,...,0} {
\foreach \n in {0,...,#3} {
\draw[black, very thick] (\n*\r-6*\k+90: 2in) -- (\n*\r+\r-6*\k+90: 2in);}
}
\end{tikzpicture}
}

\newcommand{\modcirclenk}[6]{
\begin{tikzpicture}[scale=#6, every node/.style={scale=#6}]
\draw[blue, thick] (0,0) circle (2in);
\foreach \x in {6,12,...,360} {
\draw (\x: 2in) -- (\x: 2.1in);}
\foreach \x in {0,90,180,270} {
\draw[red, ultra thick] (\x: 2in) -- (\x: 2.1in);
\draw[red, ultra thick, ->] (\x: 2.3in) -- (\x: 2.5in);
}
\foreach \x in {30,60,120,150,210,240,300,330} {
\draw[blue, thick] (\x: 2in) -- (\x: 2.1in);
\draw[blue, thick, ->] (\x: 2.3in) -- (\x: 2.4in);
}
\draw[dotted] (0:2in) -- (180:2in);
\draw[dotted] (90:2in) -- (270:2in);
\foreach \n/\t in {0/\redbf{0},6/7,12/3,18/4,24/9,30/\bluebf{5},36/4,42/1,48/3,54/8,60/\bluebf{5},66/3,72/2,78/1,84/1,90/\redbf{0},96/1,102/9,108/2,114/7,120/\bluebf{5},126/2,132/3,138/9,144/4,150/\bluebf{5},156/9,162/6,168/3,174/3,180/\redbf{0},186/3,192/7,198/6,204/1,210/\bluebf{5},216/6,222/9,228/7,234/2,240/\bluebf{5},246/7,252/8,258/9,264/9,270/\redbf{0},276/9,282/1,288/8,294/3,300/\bluebf{5},306/8,312/7,318/1,324/6,330/\bluebf{5},336/1,342/4,348/7,354/7} {
\draw (\n: 2.2in) node{\t};}

\def\r{-#1*6}
\foreach \k in {#2,...,#3} {
\foreach \n in {#4,...,#5} {
\draw[red] (\n*\r-6*\k+90: 2in) -- (\n*\r+\r-6*\k+90: 2in);}
}
\end{tikzpicture}
}

\begin{document}

\title{\textcolor{red}{\textbf{Tantalizing properties of subsequences of the Fibonacci sequence modulo 10}}}

\author{
Dan Guyer \\  
\href{mailto:guyerdm7106@uwec.edu}{\nolinkurl{guyerdm7106@uwec.edu}}
\and
aBa Mbirika\\
\href{mailto:mbirika@uwec.edu}{\nolinkurl{mbirika@uwec.edu}}
\and
Miko Scott\\
\href{mailto:mbscott11@outlook.com}{\nolinkurl{mbscott11@outlook.com}}
}

\date{}

\makeatletter
\newcommand{\subjclass}[2][2020]{%
  \let\@oldtitle\@title%
  \gdef\@title{\@oldtitle\footnotetext{#1 \emph{Mathematics Subject Classification.} #2}}%
}
\newcommand{\keywords}[1]{%
  \let\@@oldtitle\@title%
  \gdef\@title{\@@oldtitle\footnotetext{\emph{Key words and phrases.} #1.}}%
}
\makeatother


\maketitle

\begin{abstract}
The Fibonacci sequence modulo $m$, which we denote $\left(\newF_{m,n}\right)_{n=0}^\infty$ where $\newF_{m,n}$ is the Fibonacci number $F_n$ modulo $m$, has been a well-studied object in mathematics since the seminal paper by D.~D.~Wall in 1960 exploring a myriad of properties related to the periods of these sequences. Since the time of Lagrange it has been known that $\Fmn$ is periodic for each $m$. We examine this sequence when $m=10$, yielding a sequence of period length 60. In particular, we explore its subsequences composed of every $r^{\mathrm{th}}$ term of $\Ften$ starting from the term $\F{k}$ for some $0 \leq k \leq 59$. More precisely we consider the subsequences $\Ftensub$, which we show are themselves periodic and whose lengths divide 60. Many intriguing properties reveal themselves as we alter the $k$ and $r$ values. For example, for certain $r$ values the corresponding subsequences surprisingly obey the Fibonacci recurrence relation; that is, any two consecutive subsequence terms sum to the next term modulo 10. Moreover, for all $r$ values relatively prime to 60, the subsequence $\Ftensub$ coincides exactly with the original parent sequence $\Ften$ (or a cyclic shift of it) running either forward or reverse. We demystify this phenomena and explore many other tantalizing properties of these subsequences.
\end{abstract}

\tableofcontents

%%%%%%%%%%%%%%%%%%%%%%%%%%%%%%%%%%%%%
%%%%%%%%%%%%%%%%%%%%%%%%%%%%%%%%%%%%%
%%%%%%%%%    SECTION 1    %%%%%%%%%%%
%%%%%%%%%%%%%%%%%%%%%%%%%%%%%%%%%%%%%
%%%%%%%%%%%%%%%%%%%%%%%%%%%%%%%%%%%%%

\section{Introduction}\label{sec:Introduction}
A rich source of research on the Fibonacci sequence $\FibSeq$ over the past 60 years has been the sequence  reduced modulo $m$, which we denote $\Fmn$ where $\newF_{m,n}$ equals the Fibonacci number $F_n$ modulo $m$. For example, the first 17 terms of the Fibonacci sequence $\FibSeq$ and the corresponding reduced modulo 8 sequence $\left(\newF_{8,n}\right)_{n=0}^\infty$ is as follows:
\begin{center}
    \begin{tabular}{|c||c|c|c|c|c|c|c|c|c|c|c|c|c|c|c|c|c|}
    \hline
        $n$ & 0 & 1 & 2 & 3 & 4 & 5 & 6 & 7 & 8 & 9 & 10 & 11 & 12 & 13 & 14 & 15 & 16\\ \hline\hline
        $F_n$ & 0 & 1 & 1 & 2 & 3 & 5 & 8 & 13 & 21 & 34 & 55 & 89 & 144 & 233 & 377 & 610 & 987 \\ \hline
        $\newF_{8,n}$ & \bluebf{0} & \bluebf{1} & \bluebf{1} & \bluebf{2} & \bluebf{3} & \bluebf{5} & \bluebf{0} & \bluebf{5} & \bluebf{5} & \bluebf{2} & \bluebf{7} & \bluebf{1} & 0 & 1 & 1 & 2 & 3 \\ \hline
    \end{tabular}
\end{center}
It is clear that this modulo 8 sequence repeats at the start of the second occurrence of 0, 1, and hence is periodic of length 12 as shown above in bolded blue. As far back as the $18^{\mathrm{th}}$ century, Lagrange was aware of these reduced Fibonacci sequences and knew that $\Fmn$ is periodic for every $m$. In a ground-breaking paper in 1960, Wall introduced these so-called \textit{Pisano periods}, the length of the repeated Fibonacci sequence pattern modulo $m$ and was the first to prove a number of fundamental results regarding these periods~\cite{Wall1960}. And since then, much research has been done on this topic of the Fibonacci sequence modulo $m$ from a variety of perspectives~\cite{Robinson1963, Fulton-Morris1969, Erhlich1989, Renault1996, Gupta_et_al2012}.

\begin{figure}[H]
\centering
\resizebox{3in}{!}
{
\begin{tikzpicture}
%%%%%%%% circle %%%%%%%%%
\draw[blue, thick] (0,0) circle (2in);
%%%%%%%% tick marks %%%%%%%%%
\foreach \x in {6,12,...,360} {
\draw (\x: 2in) -- (\x: 2.1in);}
%%%%%%%% fib mod 10 labels %%%%%%%%%
\foreach \x/\t in {0/0,6/7,12/3,18/4,24/9,30/5,36/4,42/1,48/3,54/8,60/5,66/3,72/2,78/1,84/1,90/0,96/1,102/9,108/2,114/7,120/5,126/2,132/3,138/9,144/4,150/5,156/9,162/6,168/3,174/3,180/0,186/3,192/7,198/6,204/1,210/5,216/6,222/9,228/7,234/2,240/5,246/7,252/8,258/9,264/9,270/0,276/9,282/1,288/8,294/3,300/5,306/8,312/7,318/1,324/6,330/5,336/1,342/4,348/7,354/7} {
\draw (\x: 2.2in) node{\t};}
%%%%%%%% diagonals %%%%%%%%%
\foreach \x in {1,...,6} {
\draw[red] (\x*30: 2in) -- (\x*30+180: 2in);}
%%%%%%%% angle labels %%%%%%%%%
\foreach \x/\t in {0/90,1/60,2/30,3/0,4/330,5/300,6/270,7/240,8/210,9/180,10/150,11/120} {
\draw (\x*30: 1in) node[fill=white]{$\text{\t}^\circ$};}
\end{tikzpicture}
}
\vspace{-.15in}
\caption{The Pisano period $\left(\F{n}\right)_{n=0}^{59}$ with values equally spaced}
\label{fig:FibTen_circle}
\end{figure}

It is the goal of this paper to explore some properties of the sequence $\Fmn$ when $m=10$. This modulo 10 sequence has a Pisano period of length 60. The sequence of 60 Fibonacci values $F_0, \ldots, F_{59}$ modulo 10 can be pictorially represented by equally spacing the 60 sequence entries $\F{0}, \ldots, \F{59}$ around a circle clockwise starting with $\F{0}$ placed at the top at 0 degrees as in Figure~\ref{fig:FibTen_circle}. We investigate the subsequences composed of equally spaced terms of $\Ften$; that is, every $r^{\mathrm{th}}$ term of $\Ften$ starting from the term $\F{k}$ for some $0 \leq k \leq 59$. More precisely we consider the subsequences $\Ftensub$, which themselves are necessarily periodic with a period length dividing 60.

\begin{figure}[H]
\centering
%\begin{center}
    \modcirclenk{25}{3}{3}{0}{11}{.5}
%\end{center}
%\vspace{-.25in}
%\hspace{-1.5in}
\caption{Subsequence diagram corresponding to $\Ftensubfill{3}{25}$}
\label{fig:Lucas_subsequence}
\end{figure}

The advantage of viewing the parent sequence $\Ften$ in this circular manner in Figure~\ref{fig:FibTen_circle} is that it allows us to study subsequences as convex and non-convex star polygons inscribed within the circle itself. For example, if we set $k=3$ and $r=25$, then we get the star polygon $\{ \frac{12}{5} \}$ corresponding to the subsequence $\Ftensubfill{3}{25}$ shown in Figure~\ref{fig:Lucas_subsequence}. In Appendix~\ref{appendix_A}, we give a visual step-by-step construction of the subsequence diagram in this figure. By looking at the vertex labels on this star polygon (in the order that they occur) starting at the first subsequence term $\F{3} = 2$ towards the top-right of the figure, it is interesting to note that this subsequence has a length 12 period as follows: $(2, 1, 3, 4, 7, 1, 8, 9, 7, 6, 3, 9)$. This is surprisingly the complete period of the Lucas sequence modulo 10. Moreover, it satisfies a Fibonacci-like recurrence:
$\F{k+r(j-1)} + \F{k+rj} \equiv \F{k+r(j+1)} \pmod{10}$. We call subsequences that satisfy this type of recurrence a \textit{forward quasi-Fibonacci subsequence}, and we explore such patterns in Section~\ref{sec:quasi_fibonacci_subsequence}. Equivalently, there is a notion of a reverse version of such quasi-Fibonacci subsequences.

\begin{figure}[hbt!]
\begin{center}
    \modcirclenk{13}{15}{15}{0}{59}{.5}
\end{center}
\vspace{-.25in}
\caption{Subsequence diagram corresponding to $\Ftensubfill{15}{13}$}
\label{fig:complete_subsequence}
\end{figure}

However, the most tantalizing properties occur in subsequences $\Ftensub$ for which the value $r$ lies in the unit group $U(60)$ of integers modulo 60. For example, for $k=15$ and $r=13$, we get star polygon $\{ \frac{60}{13} \}$ corresponding to the subsequence $\Ftensubfill{15}{13}$ shown in Figure~\ref{fig:complete_subsequence}. The miraculous aspect of this subsequence is that its length 60 period is exactly the complete period of the Fibonacci sequence modulo 10. More precisely, despite the repetitive jumps of gap size 13 between each subsequence term, this subsequence $\Ftensubfill{15}{13}$ is exactly the parent sequence $\Ften$. We call sequences whose period is exactly that of the Fibonacci sequence modulo 10, or a cyclic shift of it, a \textit{forward complete Fibonacci subsequence}, and we explore such patterns in Sections~\ref{sec:complete_fibonacci_subsequences}. Equivalently, there is a notion of a reverse version of such complete Fibonacci subsequences.

The breakdown of the paper is as follows. In Section~\ref{sec:prelim}, we give some preliminary definitions and well-known identities used in our proofs. We also define the three types of subsequence diagrams that can occur as we vary the value $r$. In Section~\ref{sec:certain_subsequences}, we introduce the intrigue of studying subsequences $\Ftensub$ of the parent sequence $\Ften$ by examining some specific cases. This section leaves some proofs to the reader and contains avenues of possible further research for the motivated reader. In Sections~\ref{sec:quasi_fibonacci_subsequence} and \ref{sec:complete_fibonacci_subsequences}, we explore quasi-Fibonacci subsequences and complete Fibonacci subsequences, respectively. Lastly in Section~\ref{sec:open questions}, we give a variety of open problems related to our research.

%%%%%%%%%%%%%%%%%%%%%%%%%%%%%%%%%%%%%
%%%%%%%%%%%%%%%%%%%%%%%%%%%%%%%%%%%%%
%%%%%%%%%    SECTION 2    %%%%%%%%%%%
%%%%%%%%%%%%%%%%%%%%%%%%%%%%%%%%%%%%%
%%%%%%%%%%%%%%%%%%%%%%%%%%%%%%%%%%%%%

\section{Preliminaries}\label{sec:prelim}

%%%%%%%%%%%%%%%%%%%%%%%%%%%
%%%%%%%%%%%%%%%%%%%%%%%%%%%
%%%%%%%%%%%%%%%%%%%%%%%%%%%
\subsection{Definitions and identities}\label{subsec:definitions_and_identities}

\begin{definition}
The \textit{Fibonacci sequence} $\FibSeq$ and  \textit{Lucas sequence} $\LucSeq$ are defined by the recurrence relations $F_n = F_{n-1} + F_{n-2}$ and $L_n = L_{n-1} + L_{n-2}$, respectively,
with initial conditions $F_0 = 0$, $F_1 = 1$, $L_0 = 2$, and $L_1 = 1$.
\end{definition}

Below we give some well-known identities, which we use in this paper. Proofs of these results can be found in numerous sources (for example, Koshy~\cite{Koshy2001} or Vajda~\cite{Vajda1989}).

\begin{proposition}\label{prop:fundamental_identities}
For all $n,m \in \mathbb{Z}$, the following five identities hold:
\begin{align}
    F_{-n} &= (-1)^{n+1} F_n \label{eq:fund_identity_1}\\
    2\mbox{ divides } F_n &\Leftrightarrow 3 \mbox{  divides } n \label{eq:fund_identity_2a}\\
    5\mbox{ divides } F_n &\Leftrightarrow 5 \mbox{  divides } n \label{eq:fund_identity_2b}\\
    L_n &= F_{n-1} + F_{n+1}\label{eq:fund_identity_3} \\
    F_{n+m} &= F_{n-1} F_m + F_n F_{m+1} \label{eq:fund_identity_4}
\end{align}
\end{proposition}

\begin{definition}
The \textit{Fibonacci sequence modulo} $m$ is the sequence $\Fmn$ where $\newF_{m,n}$ is the least nonnegative residue of the value $F_n$ modulo $m$.
\end{definition}

\begin{example}[The Fibonacci sequence modulo $10$]\label{exam:pisano_10}
Consider the following table of Fibonacci numbers $F_n$ and their corresponding values modulo 10 for $n=0, \ldots, 15$. The usage of the yellow highlighting and red bold font reflect the results of Lemma~\ref{lem:Pisano_divis_by_5} which states that $\F{n} = 5$ if and only if $5$ divides $n$ and 15 does not divide $n$, and moreover, $\F{n}=0$ if and only if 15 divides $n$.
\begin{center}
\begin{tabular}{|c||c|c|c|c|c|c|c|c|c|c|c|c|c|c|c|c|}
\hline
$n$ & 0 & 1 & 2 & 3 & 4 & 5 & 6 & 7 & 8 & 9 & 10 & 11 & 12 & 13 & 14 & 15\\ \hline\hline
$F_n$ & 0 & 1 & 1 & 2 & 3 & 5 & 8 & 13 & 21 & 34 & 55 & 89 & 144 & 233 & 377 & 610\\ \hline
$\F{n}$ & \hl{0} & 1 & 1 & 2 & 3 & \redbf{5} & 8 & 3 & 1 & 4 & \redbf{5} & 9 & 4 & 3 & 7 & \hl{0}\\ \hline
\end{tabular}
\end{center}
Since the $F_n$ values grow large very quickly, in the tables below we give only the values $\F{n}$ for $n=16, \ldots, 60$.
\begin{center}
\begin{tabular}{|c||c|c|c|c|c|c|c|c|c|c|c|c|c|c|c|}
\hline
$n$ & 16 & 17 & 18 & 19 & 20 & 21 & 22 & 23 & 24 & 25 & 26 & 27 & 28 & 29 & 30\\ \hline\hline
$\F{n}$ & 7 & 7 & 4 & 1 & \redbf{5} & 6 & 1 & 7 & 8 & \redbf{5} & 3 & 8 & 1 & 9 & \hl{0}\\ \hline
\end{tabular}
\end{center}

\begin{center}
\begin{tabular}{|c||c|c|c|c|c|c|c|c|c|c|c|c|c|c|c|}
\hline
$n$ & 31 & 32 & 33 & 34 & 35 & 36 & 37 & 38 & 39 & 40 & 41 & 42 & 43 & 44 & 45\\ \hline\hline
$\F{n}$ & 9 & 9 & 8 & 7 & \redbf{5} & 2 & 7 & 9 & 6 & \redbf{5} & 1 & 6 & 7 & 3 & \hl{0}\\ \hline
\end{tabular}
\end{center}

\begin{center}
\begin{tabular}{|c||c|c|c|c|c|c|c|c|c|c|c|c|c|c|c|}
\hline
$n$ & 46 & 47 & 48 & 49 & 50 & 51 & 52 & 53 & 54 & 55 & 56 & 57 & 58 & 59 & 60\\ \hline
$\F{n}$ & 3 & 3 & 6 & 9 & \redbf{5} & 4 & 9 & 3 & 2 & \redbf{5} & 7 & 2 & 9 & 1 & \hl{0}\\ \hline
\end{tabular}
\end{center}
It is clear that $F_{60} \equiv 0 \pmod{10}$ and $F_{61} \equiv 1 \pmod{10}$; moreover, we also have $r=60$ being the smallest positive value $r$ such that $F_r \equiv 0 \pmod{10}$ and $F_{r+1} \equiv 1 \pmod{10}$. Observe that the reduced sequence $\Ften$ is periodic with period length 60. Using Definition~\ref{def:Pisano} to follow, we will say $\pi(10) = 60$.
\end{example}

The previous example shows that modulo 10, the Fibonacci sequence is periodic. This however is not unique to the value 10. The following proposition proves that the sequence $\Fmn$ is periodic for all moduli values $m$.

\begin{proposition}\label{prop:pisano_seq_is_periodic}
The sequence $\Fmn$ is periodic.
\end{proposition}

\begin{proof}
For ease of notation let $\mathcal{S}$ denote $\Fmn$. Clearly $\mathcal{S}$ is completely determined by any two adjacent terms since $\mathcal{S}$ obeys the Fibonacci recurrence $\newF_{m,n} \equiv \newF_{m,n-1} + \newF_{m,n-2} \pmod{10}$ for all $n \geq 2$. Hence $\mathcal{S}$ is periodic if and only if any two ordered pairs of two adjacent terms coincide. But since all the sequence terms are elements of the set $\{0, 1, \ldots, m-1\}$, there are $m^2$ possible ordered pairs $i,j$ where $0 \leq i,j \leq m-1$. Indeed no two adjacent terms will ever be $0,0$ so there are at most $m^2 - 1$ possible distinct ordered pairs. Hence by the pigeonhole principle no sequence $\mathcal{S}$ can exceed length $2(m^2-1)$ without repeating a pair of adjacent terms. Thus $\mathcal{S}$ is periodic.
\end{proof}

\begin{definition}\label{def:Pisano}
The \textit{Pisano period} of the sequence $\Fmn$ is the smallest  positive integer $r$ such that the congruences $F_r \equiv 0 \pmod{m}$ and $F_{r+1} \equiv 1 \pmod{m}$ hold; that is, $\newF_{m,r}=0$ and $\newF_{m,r+1}=1$. We denote this period $\pi(m)$.
\end{definition}

%%%%%%%%%%%%%%%%%%%%%%%%%%%
%%%%%%%%%%%%%%%%%%%%%%%%%%%
%%%%%%%%%%%%%%%%%%%%%%%%%%%
\subsection{The three types of subsequence diagrams}\label{subsec:three_types_of_diagrams}

For each $k$ and $r$ value, the subsequence $\Ftensub$ of the sequence $\Ften$ corresponds to a unique subsequence diagram in the $\Ften$-circle (see Figure~\ref{fig:FibTen_circle}), and this subsequence diagram depends only only the jump size $r$. The value $r$ determines whether the subsequence diagram is either a convex regular $n$-gon or a non-convex star polygon with $n$ vertices, where $n$ is both the period length and the number of vertices. The value $n$ is easily determined by the value $r$ as Proposition~\ref{prop:computing_number_of_vertices} yields. Before we prove this proposition, we first remind the reader of the definition of a star polygon.

\begin{definition}\label{def:star_polygon}
A \textit{star polygon $\{\frac{p}{q}\}$} with $p,q \in \mathbb{N}$ and $p \geq 3$ is the graph formed by connecting every $q^{\mathrm{th}}$ point out of $p$ regularly spaced points lying on a circumference of a circle. A non-convex polygon arises if $\gcd(p,q)=1$ and $q \neq 1$.  A convex regular $p$-gon arises if $q=1$.
\end{definition}

\begin{example}
Consider the three different subsequence diagrams corresponding to the sequences $\Ftensub$ with starting value $k=3$ and the jump sizes varying over the three values $r=12, 25,$ and $17$, respectively. In Figure~\ref{fig:3_subseq_diagrams}, we draw these diagrams. The subsequence $\Ftensubfill{3}{12}$ corresponds to the star polygon $\{\frac{5}{1}\}$ (that is, a regular 5-gon) in the far left of the figure. The subsequence $\Ftensubfill{3}{25}$ corresponds to the star polygon $\{\frac{12}{5}\}$ in the middle of the figure. Lastly, the subsequence $\Ftensubfill{3}{17}$ corresponds to the star polygon $\{\frac{60}{17}\}$ in the far right of the figure. In Theorem~\ref{thm:3_types_of_diagrams}, we give conditions on the value $r$ that determine the diagram type.
\end{example}

\begin{figure}[h]
\begin{center}
    $\underset{\modcirclenk{12}{3}{3}{0}{4}{.3}}{\textbf{Type 1}}$
    $\underset{\modcirclenk{25}{3}{3}{0}{11}{.3}}{\textbf{Type 2}}$
    $\underset{\modcirclenk{17}{3}{3}{0}{59}{.3}}{\textbf{Type 3}}$
\end{center}
%\vspace{-.25in}
\caption{Subsequence diagrams for $\Ftensubfill{3}{12}$, $\Ftensubfill{3}{25}$, and $\Ftensubfill{3}{17}$}
\label{fig:3_subseq_diagrams}
\end{figure}

\begin{proposition}\label{prop:computing_number_of_vertices}
Given a jump size of $r \in \{1, \ldots, 60\}$, the subsequence diagram corresponding to $\Ftensub$ has $n$ vertices where $n = \frac{\lcm(r,60)}{r}$. Equivalently, we have $n = \frac{60}{\gcd(r,60)}$.
\end{proposition}

\begin{proof}
Since the subsequence diagram corresponding to $\Ftensub$ has the same shape up to rotation regardless of the value $k$, it suffices to consider the subsequence $\FtensubZeroK$. Let $n$ denote the number of vertices of the subsequence diagram. Then these $n$ vertices are connected by edges in the following sequence of points on the $\Ften$-circle: $\F{0},\F{r}, \F{2r}, \ldots, \F{nr}$, where $n$ is the least positive integer such that $nr$ is a multiple of 60. Hence $nr$ must be the smallest value that is both a multiple of $r$ and $60$, and thus $nr = \lcm(r,60)$. Solving for $n$ we get $n = \frac{\lcm(r,60)}{r}$ as desired. Moreover since $\gcd(r,60) \cdot \lcm(r,60) = 60r$ holds trivially, dividing by $r$ we get
$$\gcd(r,60) \cdot \frac{\lcm(r,60)}{r} = \frac{60 r}{r}.$$
Substituting $n$ for $\frac{\lcm(r,60)}{r}$ in the  latter yields $\gcd(r,60) \cdot n = 60$, and this implies the equivalent formula $n = \frac{60}{\gcd(r,60)}$.
\end{proof}

\begin{corollary}\label{cor:periodicity_of_subsequence}
The subsequence $\Ftensub$ is periodic of length $\frac{60}{\gcd(r,60)}$. In particular, the period of $\Ftensub$ divides the period $\pi(10)$ of $\Ften$, where $\pi(10) = 60$.
\end{corollary}

\begin{proof}
Consider the values $\F{k}, \F{k+r}, \F{k+2r}, \ldots, \F{k+(n-1)r}$ where $n = \frac{60}{\gcd(r,60)}$. We claim that these first $n$ values of the subsequence $\Ftensub$ give a complete period of the subsequence $\Ftensub$. This follows since in the subsequence diagram if we travel from the vertex corresponding to the $\F{k+(n-1)r}$-value to the vertex corresponding to the $\F{k+nr}$-value, then we return back to the starting vertex corresponding to the first sequence term $\F{k}$. Hence $\Ftensub$ is periodic and has length $n = \frac{60}{\gcd(r,60)}$. Clearly $n$ divides the period $\pi(10)=60$.
\end{proof}

\begin{remark}\label{rem:r_versus_n_minus_r}
For a given $r \in \{1, 2, \ldots 59\}$, the subsequence diagrams for $\Ftensubfill{k}{r}$ and $\Ftensubfill{k}{(60-r)}$ will coincide. However, the actual terms of the subsequences will be reversals of each other. For example, the subsequence $\Ftensubfill{3}{12}$ in the far left of Figure~\ref{fig:3_subseq_diagrams} is the sequence repeating the period $(2, 0, 8, 6, 4)$. Whereas,  the subsequence $\Ftensubfill{3}{(60-12)}$ would be the sequence repeating the period $(2, 4, 6, 8, 0, 2)$.
\end{remark}

Before we prove Theorem~\ref{thm:3_types_of_diagrams}, which classifies the three types of subsequence diagrams, we prove a useful lemma used in the proof.

\begin{lemma}\label{lem:mini_gcd_result}
Let $a,b \in \mathbb{N}$. Then the following identity holds:
$$ \gcd\left( \frac{a}{\gcd(a,b)}, \frac{b}{\gcd(a,b)} \right) = 1.$$
\end{lemma}

\begin{proof}
Let $d = \gcd(a,b)$. It suffices to write 1 as an integer linear combination of $\frac{a}{d}$ and $\frac{b}{d}$. Observe that by B\'{e}zout's identity\footnote{Historical remark: Though B\'{e}zout is attributed to this identity in the $18^{\mathrm{th}}$ century, this result was known by the Indian mathematician \={A}ryabha\d{t}a in the $5^{\mathrm{th}}$ century.}, we know there exists integers $x,y \in \mathbb{Z}$ such that $ax + by = d$. Dividing both sides by $d$, we get $\frac{a}{d}x + \frac{b}{d}y = 1$. Hence $\frac{a}{d}$ and $\frac{b}{d}$ are relatively prime and the result holds.
\end{proof}

\begin{remark}
Since the identity $\gcd(r,60) = \gcd(60-r,60)$ always holds, the values $n$ and $q$ in Theorem~\ref{thm:3_types_of_diagrams} will also coincide whether we have a jump size of $r$ or a jump size of $60-r$. So we consider subsequences of the latter two jump sizes to be of the same type. In particular for $r$ equal to 1 or 59, we take the convention of considering the subsequence diagrams for $\Ftensub$ to be of Type 3 even though they appear as convex polygons.
\end{remark}

\begin{theorem}\label{thm:3_types_of_diagrams}
Let $r \in \{1, \ldots, 59\}$. Then for all starting $k$-values, the subsequence $\Ftensub$ corresponds to either a convex star polygon $\{\frac{n}{q}\}$ (that is, $q=1$ so we have a regular $n$-gon) or a non-convex star polygon $\{\frac{n}{q}\}$, where in both cases $n = \frac{60}{\gcd(r,60)}$ and $q = \frac{r}{\gcd(r,60)}$. In particular, we get the following three types of subsequence diagrams dependent on the value $r$:
\begin{itemize}
    \item {\normalfont{\textbf{Type 1:}}} If $r$ or $60-r$ divides $60$, then a regular $n$-gon arises.
    \item {\normalfont{\textbf{Type 2:}}} If neither $r$ nor $60-r$ divides $60$ and $\gcd(r,60) > 1$, then a non-convex star polygon $\{\frac{n}{q}\}$ arises.
    \item {\normalfont{\textbf{Type 3:}}} If $\gcd(r,60) = 1$ and $r$ is neither $1$ nor $59$, then a non-convex star polygon $\{\frac{60}{q}\}$ arises. If $r$ equals $1$ or $59$, then a convex star polygon $\{\frac{60}{q}\}$ arises.
\end{itemize}
\end{theorem}

\begin{proof}
Let $r \in \{1, \ldots, 59\}$. By Proposition~\ref{prop:computing_number_of_vertices}, the number of vertices of the corresponding subsequence diagram of $\Ftensub$ is $n = \frac{60}{\gcd(r,60)}$. Observe that each edge in the subsequence diagram corresponds to exactly two consecutive terms $\F{k + r j_0}$ and $\F{k + r (j_0 + 1)}$ in the period $\left(\newF_{10,k+rj}\right)_{j=0}^{n-1}$ of the subsequence $\Ftensub$ for some $j_0 \in \{0, 1, \ldots, n-1\}$. Hence the diagram produced is either a convex or non-convex star polygon $\{\frac{n}{q}\}$ formed by connecting every $q^{\mathrm{th}}$ point out of $n$ regularly spaced points on the 60-point $\Ften$-circle. We claim that this $q$ value is $\frac{r}{\gcd(r,60)}$. To show this, first observe that since there are exactly 60 points on the $\Ften$-circle and the subsequence diagram is inscribed inside this circle, the distance between each vertex in the subsequence diagram is $\frac{60}{n}$ points apart in the $\Ften$-circle. Moreover, any given vertex $\F{k + r j_0}$ in the subsequence diagram is adjacent to vertex $\F{k + r (j_0 + 1)}$, and these two adjacent vertices are $r$ units apart in the $\Ften$-circle. Thus there are exactly $\frac{r}{60/n}$ vertices of the subsequence diagram contained in this interval of $r$ points. Hence $q$ equals $\frac{r}{60/n}$. Since $n = \frac{60}{\gcd(r,60)}$, we have the sequence of equalities yielding the desired formula for $q$ as follows:
$$ q = \frac{r}{\sfrac{60}{n}} = \frac{r}{\sfrac{60}{\left(\frac{60}{\gcd(r,60)}\right)}} = \frac{r}{\gcd(r,60)}.$$

Using these two formulas $n = \frac{60}{\gcd(r,60)}$ and $q = \frac{r}{\gcd(r,60)}$, we can classify three types of subsequence diagrams dependent only on the jump size $r$. Recall by Definition~\ref{def:star_polygon}, a star polygon $\{\frac{p}{q}\}$ is non-convex if $\gcd(p,q) = 1$ and $q \neq 1$, and it is convex if $q=1$. There are three cases to consider.

\medskip

\noindent \textbf{Case 1}: If $r$ or $60-r$ divides 60, then $\gcd(r,60) = r$ and hence $q = 1$. Thus the star polygon $\{\frac{n}{1}\}$ produced is convex and is a regular $n$-gon. We call this Type 1.

\medskip

\noindent \textbf{Case 2}: If neither $r$ nor $60-r$ divides $60$ and $\gcd(r,60) > 1$, then $\gcd(r,60) \neq r$ and hence $q \neq 1$. By Lemma~\ref{lem:mini_gcd_result}, we have
$$ \gcd(n,q) = \gcd\left(\frac{60}{\gcd(r,60)}, \frac{r}{\gcd(r,60)}\right) = 1.$$
Thus the star polygon $\{\frac{n}{q}\}$ produced is non-convex and has less than 60 vertices since we have $n = \gcd(r,60) < 60$. We call this Type 2.

\medskip

\noindent \textbf{Case 3}: If $\gcd(r,60) = 1$ and $r$ is neither $1$ nor $59$, then $n = 60$ and $q = r$ and hence $\gcd(n,q)=1$. Also observe that $q \neq 1$ by assumption and hence $q>1$. Thus the star polygon $\{\frac{60}{q}\}$ produced is non-convex and utilizes all 60 points of the $\Ften$-circle as its vertices. We call this Type 3. On the other hand, if $r=1$ or $r=59$, then again $n=60$ and $q=r$ and hence the star polygon $\{\frac{60}{q}\}$ produced is convex and utilizes all 60 points of the $\Ften$-circle as its vertices. We also call this Type 3.
\end{proof}

\begin{remark}
Different $r$ values can yield diagrams of Type 2 with the same number of vertices. For instance, let $k=0$ and consider the $r$ values 9, 21, and 27. Then it is easily computed using the formula $n = \frac{60}{\gcd(r,60)}$ that $n = 20$ for all of these $r$ values. However, when computing the $q$ values given by the formula $q = \frac{r}{\gcd(r,60)}$, we get the corresponding $q$ values 3, 7, and 9, respectively. This yields the corresponding star polygons $\{\frac{20}{3}\}$, $\{\frac{20}{7}\}$, and $\{\frac{20}{9}\}$, respectively, shown in Figure~\ref{fig:3_diagrams_with_same_n_value}:
\begin{figure}[H]
\begin{center}
    $\underset{\modcirclenk{9}{0}{0}{0}{19}{.4}}{\bm{r=9}}$
    $\underset{\modcirclenk{21}{0}{0}{0}{19}{.4}}{\bm{r=21}}$
    $\underset{\modcirclenk{27}{0}{0}{0}{19}{.4}}{\bm{r=27}}$
\end{center}
%\vspace{-.25in}
\caption{Subsequence diagrams for $\Ftensubfill{0}{9}$, $\Ftensubfill{0}{21}$, and $\Ftensubfill{0}{27}$}
\label{fig:3_diagrams_with_same_n_value}
\end{figure}
\end{remark}

%%%%%%%%%%%%%%%%%%%%%%%%%%%%%%%%%%%%%
%%%%%%%%%%%%%%%%%%%%%%%%%%%%%%%%%%%%%
%%%%%%%%%    SECTION 3    %%%%%%%%%%%
%%%%%%%%%%%%%%%%%%%%%%%%%%%%%%%%%%%%%
%%%%%%%%%%%%%%%%%%%%%%%%%%%%%%%%%%%%%

\section{Observations on some specific subsequences}\label{sec:certain_subsequences}

In this section we examine some specific subsequences that yield convex subsequence diagrams, namely Type 1 diagrams. We omit detailed formal proofs for many of the assertions in this section as most of these results can be verified informally by brute force verification of the entries in the $\Ften$-circle. Hence we leave formal proofs to the motivated reader.

%%%%%%%%%%%%%%%%%%%%%%%%%%%
%%%%%%%%%%%%%%%%%%%%%%%%%%%
%%%%%%%%%%%%%%%%%%%%%%%%%%%
\subsection{Jump size \texorpdfstring{$r=30$}{r=30}: antipodal points}
\textcolor{white}{Consider the shape produced by a jump size of $r=30$.}
\newline
\begin{minipage}{.4\textwidth}
Given a jump size of $r=30$, we can observe that the sum of the antipodal points on the $\Ften$-circle is either 0 or 10. For instance in Figure~\ref{fig:r=30}, the bold black line represents the period $(0,0)$ of the subsequence $\Ftensubfill{0}{30}$. This sum is 0, whereas the period $(1,9)$ for the neighboring subsequence $\Ftensubfill{1}{30}$ the sum is 10. This phenomena occurs for all $\Ftensubfill{k}{30}$ whenever $k$ is not divisible by 15. This becomes evident when looking at the column values for $\F{n}$ and  $\F{n+30}$ in Table~\ref{table:antipodal_sums}.
\end{minipage}
\begin{minipage}{.6\textwidth}
\begin{figure}[H]
\begin{center}
\modcircleWithBlack{30}{29}{1}{.6}
\end{center}
%\vspace{-.25in}
\caption{$\Ftensubfill{k}{30}$ diagrams}
\label{fig:r=30}
\end{figure}
\end{minipage}

\begin{table}[H]
\begin{center}
\begin{tabular}{|c||c|c|c|c|c|c|c|c|c|c|c|c|c|c|c|c|}
\hline
$n$ & 0 & 1 & 2 & 3 & 4 & 5 & 6 & 7 & 8 & 9 & 10 & 11 & 12 & 13 & 14 & 15\\ \hline\hline
$\F{n}$ & 0 & 1 & 1 & 2 & 3 & 5 & 8 & 3 & 1 & 4 & 5 & 9 & 4 & 3 & 7 & 0\\ \hline
$\F{n+30}$ & 0 & 9 & 9 & 8 & 7 & 5 & 2 & 7 & 9 & 6 & 5 & 1 & 6 & 7 & 3 & 0\\ \hline
\end{tabular}
\end{center}

\begin{center}
\begin{tabular}{|c||c|c|c|c|c|c|c|c|c|c|c|c|c|c|}
\hline
$n$ & 16 & 17 & 18 & 19 & 20 & 21 & 22 & 23 & 24 & 25 & 26 & 27 & 28 & 29\\ \hline\hline
$\F{n}$ & 7 & 7 & 4 & 1 & 5 & 6 & 1 & 7 & 8 & 5 & 3 & 8 & 1 & 9\\ \hline
$\F{n+30}$ & 3 & 3 & 6 & 9 & 5 & 4 & 9 & 3 & 2 & 5 & 7 & 2 & 9 & 1\\ \hline
\end{tabular}
\caption{Sums of antipodal points on the $\Ften$-circle}
\label{table:antipodal_sums}
\end{center}
\end{table}

This observation leads to the main result Theorem~\ref{thm:antipodal_points_sum} of this subsection. Its proof is aided by the next two useful lemmas. This first lemma follows from the well-known results that $3$ divides $n$ if and only if $F_n$ is even, and $5$ divides $n$ if and only if $5$ divides $F_n$.

\begin{lemma}\label{lem:Pisano_divis_by_5}
In the sequence $\Ften$, the following identity holds:
\[
 \F{n} = 
  \begin{cases} 
   5 & \text{if and only if } 5 \divides n \text{ and } 15 \ndivides n,\\
   0 & \text{if and only if } 15 \divides n.
  \end{cases}
\]
\end{lemma}

This second lemma reveals an intimate connection between jump sizes in $\Ften$ that are multiples of 15, such as from $\F{n}$ to $\F{n+15j}$ for each $j \geq 0$, and the power sequence $(7^j \pmod{10})_{j=0}^\infty$ which is a periodic sequence of length 4, namely $(7,9,1,3,7,9,1,3,\ldots)$. Its proof relies on Lemma~\ref{lem:Vorobiev}, which is a modulo 10 analogue of Identity~\eqref{eq:fund_identity_4} of Proposition~\ref{prop:fundamental_identities}.

\begin{lemma}\label{lem:powers_of_7}
Set $n \geq 0$. Then for all $j \geq 0$, we have $\F{n+15j}\equiv 7^j \cdot \F{n}\pmod{10}$.
\end{lemma}

\begin{proof}
We induct on $j$. Clearly for $j=0$, we have $\F{n}\equiv 7^0 \cdot \F{n}\pmod{10}$ so the base case holds. Suppose for some $k \geq 0$ that the congruence $\F{n+15k}\equiv 7^k \cdot \F{n}\pmod{10}$ holds. It suffices to show that $\F{n+15(k+1)}\equiv 7^{k+1} \cdot \F{n}\pmod{10}$ holds. Observe the sequence of equalities and congruences
\begin{align*}
    \F{n+15(k+1)} &= \F{(n+15k) + 15}\\
        &\equiv \F{(n+15)-1} \F{15} + \F{n+15k} \F{16} \pmodd{10} &\text{by Lemma~\ref{lem:Vorobiev}}\\
        &\equiv 7 \cdot \F{n+15k} \pmodd{10} &\text{since $\F{15}=0$ and $\F{16}=7$}\\
        &\equiv 7 \cdot \left( 7^k \cdot \F{n} \right) \pmodd{10} &\text{by the induction hypothesis}\\
        &\equiv 7^{k+1} \cdot \F{n} \pmodd{10},
\end{align*}
as desired. We conclude that $\F{n+15j}\equiv 7^j \cdot \F{n}\pmod{10}$ for all $j \geq 0$.
\end{proof}

\begin{theorem}\label{thm:antipodal_points_sum}
In the sequence $\Ften$, the following identity holds:
\[
 \F{n} + \F{n+30} = 
  \begin{cases} 
   0 & \text{if } 15 \mbox{ divides } n,\\
   10 & \text{otherwise}.
  \end{cases}
\]
\end{theorem}

\begin{proof}
By Lemma~\ref{lem:Pisano_divis_by_5}, the claim clearly holds if 15 divides $n$, so assume 15 does not divide $n$. By Lemma~\ref{lem:powers_of_7}, we know 
$\F{n+15(2)}\equiv 7^2 \cdot \F{n}\pmod{10}$. Hence it follows that $$\F{n+30}\equiv 9 \cdot \F{n} \equiv -1 \cdot \F{n}\pmodd{10}.$$
%\begin{align*}
%\F{n+30}\equiv 9 \cdot \F{n} \equiv -1 \cdot \F{n}\pmod{10}.
%\end{align*}
And thus $\F{n+30} + \F{n}\equiv 0 \pmod{10}$.
Since 15 does not divide $n$ then both $\F{n}$ and $\F{n+30}$ are nonzero, and so 0 $< \F{n}$, $\F{n+30} <$ 10 and 
$\F{n+30} + \F{n}\equiv 0 \pmod{10}$. We conclude that $\F{n+30} + \F{n} = 10$ as desired.
\end{proof}

%%%%%%%%%%%%%%%%%%%%%%%%%%%
%%%%%%%%%%%%%%%%%%%%%%%%%%%
%%%%%%%%%%%%%%%%%%%%%%%%%%%
\subsection{Jump size \texorpdfstring{$r=15$}{r=15}: squares}
\textcolor{white}{Consider the shape produced by a jump size of $r=15$.}
\newline
\begin{minipage}{.4\textwidth}
Given a jump size of $r=15$, we can observe that each subsequence $\Ftensubfill{k}{15}$ is periodic of length 4. For instance in Figure~\ref{fig:r=15}, the bold black square represents the period $(0,0,0,0)$ of the subsequence $\Ftensubfill{0}{15}$. In general, consider the 4-tuples
$$\mathcal{S}_k := \left( \F{k+15j} \; | \; 0 \leq j \leq 3 \right)$$
for $k \in \{0, 1, \ldots, 14\}$. We list the elements of each $\mathcal{S}_k$ tuple below in the order that the $\F{k+15j}$ appear clockwise in the circle. Each subsequence $\Ftensubfill{k}{15}$ is an infinitely repeating sequence of exactly one tuple below or a cyclic shift of one tuple below.
\end{minipage}
\begin{minipage}{.6\textwidth}
\begin{figure}[H]
\begin{center}
\modcircleWithBlack{15}{14}{3}{.65}
\end{center}
%\vspace{-.25in}
\caption{ $\Ftensubfill{k}{15}$ diagrams}
\label{fig:r=15}
\end{figure}
\end{minipage}

\bigskip

\begin{center}
\begin{tabular}{|c||c|}
\hline
$\mathcal{S}_0$ & 0 0 0 0\\ \hline
$\mathcal{S}_1$ & 1 7 9 3\\ \hline
$\mathcal{S}_2$ & 1 7 9 3\\ \hline
$\mathcal{S}_3$ & 2 4 8 6\\ \hline
$\mathcal{S}_4$ & 3 1 7 9\\ \hline
\end{tabular}
\hspace{.75in}
\begin{tabular}{|c||c|}
\hline
$\mathcal{S}_5$ & 5 5 5 5\\ \hline
$\mathcal{S}_6$ & 8 6 2 4\\ \hline
$\mathcal{S}_7$ & 3 1 7 9\\ \hline
$\mathcal{S}_8$ & 1 7 9 3\\ \hline
$\mathcal{S}_9$ & 4 8 6 2\\ \hline
\end{tabular}
\hspace{.75in}
\begin{tabular}{|c||c|}
\hline
$\mathcal{S}_{10}$ & 5 5 5 5\\ \hline
$\mathcal{S}_{11}$ & 9 3 1 7\\ \hline
$\mathcal{S}_{12}$ & 4 8 6 2\\ \hline
$\mathcal{S}_{13}$ & 3 1 7 9\\ \hline
$\mathcal{S}_{14}$ & 7 9 3 1\\ \hline
\end{tabular}
\end{center}
For any $k \in \{0, 1, \ldots, 59\}$, we make the following observations:
\begin{itemize}
\item If $\gcd(k,15)=1$, then the $\mathcal{S}_k$ tuples contain the numbers 1, 7, 9, 3, in that order or some cyclic shift.
\item If $\gcd(k,15)=3$, then the $\mathcal{S}_k$ tuples contain the numbers 2, 4, 8, 6, in that order or some cyclic shift. The fact that these tuple entries are even is clear by the well-known result that $F_n$ is even if and only if 3 divides $n$.
\item If $\gcd(k,15)=5$, then the $\mathcal{S}_k$ tuples contain only 5's. This is clear by Lemma~\ref{lem:Pisano_divis_by_5}.
\item If $\gcd(k,15)=15$, then the $\mathcal{S}_k$ tuples contain only 0's. This is clear by Lemma~\ref{lem:Pisano_divis_by_5}.
\item The sum of the entries in $\mathcal{S}_k$ equals 0 if $\gcd(k,15)=15$ and equals 20 if $\gcd(k,15) \neq 15$.
\item Powers of 0, 2, 5, and 7 modulo 10 yield the $\mathcal{S}_k$ tuples in the following sense:
\begin{itemize}
    \item If $\gcd(k,15)=1$, then $\Ftensubfill{k}{15} = (7^j \pmod{10})_{j=N}^\infty$ for some $N = 0,1,2,3$.
    \item If $\gcd(k,15)=3$, then $\Ftensubfill{k}{15} = (2^j \pmod{10})_{j=N}^\infty$ for some $N = 0,1,2,3$.
    \item If $\gcd(k,15)=5$, then $\Ftensubfill{k}{15} = (5^j \pmod{10})_{j=0}^\infty$.
    \item If $\gcd(k,15)=15$, then $\Ftensubfill{k}{15} = (0^j \pmod{10})_{j=0}^\infty$.
\end{itemize}
\end{itemize}
We leave it to the interested reader to formally prove the assertions above.

%%%%%%%%%%%%%%%%%%%%%%%%%%%
%%%%%%%%%%%%%%%%%%%%%%%%%%%
%%%%%%%%%%%%%%%%%%%%%%%%%%%
\subsection{Jump size \texorpdfstring{$r=12$}{r=12}: pentagons}
\textcolor{white}{Consider the shape produced by a jump size of $r=12$.}
\newline
\begin{minipage}{.4\textwidth}
Given a jump size of $r=15$, we can observe that each subsequence $\Ftensubfill{k}{15}$ is periodic of length 5. For instance in Figure~\ref{fig:r=12}, the bold black pentagon represents the period $(0,4,8,2,6)$ of the subsequence $\Ftensubfill{0}{12}$. In general, consider the 5-tuples
$$\mathcal{P}_k := \left( \F{k+12j} \; | \; 0 \leq j \leq 4 \right)$$
for $k \in \{0, 1, \ldots, 11\}$. We list the elements of each $\mathcal{P}_k$ tuple below in the order that the $\F{k+12j}$ appear clockwise in the circle. Each subsequence $\Ftensubfill{k}{12}$ is an infinitely repeating sequence of exactly one tuple below or a cyclic shift of one tuple below.
\end{minipage}
\begin{minipage}{.6\textwidth}
\begin{figure}[H]
\begin{center}
\modcircleWithBlack{12}{11}{4}{.65}
\end{center}
%\vspace{-.25in}
\caption{ $\Ftensubfill{k}{12}$ diagrams}
\label{fig:r=12}
\end{figure}
\end{minipage}

\bigskip

\begin{center}
\begin{tabular}{|c||c|}
\hline
$\mathcal{P}_0$ & 0 4 8 2 6\\ \hline
$\mathcal{P}_1$ & 1 3 5 7 9\\ \hline
$\mathcal{P}_2$ & 1 7 3 9 5\\ \hline
$\mathcal{P}_3$ & 2 0 8 6 4\\ \hline
$\mathcal{P}_4$ & 3 7 1 5 9\\ \hline
$\mathcal{P}_5$ & 5 7 9 1 3\\ \hline
\end{tabular}
\hspace{.75in}
\begin{tabular}{|c||c|}
\hline
$\mathcal{P}_6$ & 8 4 0 6 2\\ \hline
$\mathcal{P}_7$ & 3 1 9 7 5\\ \hline
$\mathcal{P}_8$ & 1 5 9 3 7\\ \hline
$\mathcal{P}_9$ & 4 6 8 0 2\\ \hline
$\mathcal{P}_{10}$ & 5 1 7 3 9\\ \hline
$\mathcal{P}_{11}$ & 9 7 5 3 1\\ \hline
\end{tabular}
\end{center}

\noindent For any $k \in \{0, 1, \ldots, 59\}$, we make the following observations:
\begin{itemize}
\item If $k \equiv 1 \mbox{ or } 5 \pmod{12}$, then the $\mathcal{P}_k$ tuples contain the numbers 1, 3, 5, 7, 9, in that order or some cyclic shift. If $k \equiv 7 \mbox{ or } 11 \pmod{12}$, then the $\mathcal{P}_k$ tuples contain the numbers 9, 7, 5, 3, 1, in that order or some cyclic shift.
\item If $k \equiv 2 \mbox{ or } 10 \pmod{12}$, then the $\mathcal{P}_k$ tuples contain the numbers 1, 7, 3, 9, 5, in that order or some cyclic shift. If $k \equiv 4 \mbox{ or } 8 \pmod{12}$, then the $\mathcal{P}_k$ tuples contain the numbers 5, 9, 3, 7, 1, in that order or some cyclic shift.
\item If $k \equiv 3 \pmod{12}$, then the $\mathcal{P}_k$ tuples contain the numbers 8, 6, 4, 2, 0, in that order or some cyclic shift. If $k \equiv 9 \pmod{12}$, then the $\mathcal{P}_k$ tuples contain the numbers 0, 2, 4, 6, 8, in that order or some cyclic shift.
\item If $k \equiv 0 \pmod{12}$, then the $\mathcal{P}_k$ tuples contain the numbers 0, 4, 8, 2, 6, in that order or some cyclic shift. If $k \equiv 6 \pmod{12}$, then the $\mathcal{P}_k$ tuples contain the numbers 6, 2, 8, 4, 0, in that order or some cyclic shift.
\item Consequently, the sum of the $\mathcal{P}_k$ values equals 20 if $k \equiv 0, 3, 6, 9 \pmod{12}$ and equals 25 if $k \equiv 1, 2, 4, 5, 7, 8, 10, 11 \pmod{12}$.
\end{itemize}
We leave it to the interested reader to formally prove the assertions above.

%%%%%%%%%%%%%%%%%%%%%%%%%%%
%%%%%%%%%%%%%%%%%%%%%%%%%%%
%%%%%%%%%%%%%%%%%%%%%%%%%%%
\subsection{Jump size \texorpdfstring{$r=5$}{r=5}: dodecagons}\label{subsec:dodecagon}
\textcolor{white}{Consider the shape produced by a jump size of $r=5$.}
\newline
\begin{minipage}{.4\textwidth}
Given a jump size of $r=5$, we can observe that each subsequence $\Ftensubfill{k}{5}$ is periodic of length 12. In Figure~\ref{fig:r=5}, a jump size of $r=5$ highlights a regular 12-gon, or dodecagon. For instance in the figure, the bold black dodecagon represents the period $(0,5,5,0,5,5,0,5,5,0,5,5)$ of the subsequence $\Ftensubfill{0}{5}$. In general, consider the 12-tuples
$$\mathcal{D}_k := \left( \F{k+5j} \; | \; 0 \leq j \leq 11 \right)$$
for $k \in \{0, 1, 2, 3, 4\}$. We list the elements of each $\mathcal{D}_k$ tuple below in the order that the $\F{k+5j}$ appear clockwise in the circle. Each subsequence $\Ftensubfill{k}{5}$ is an infinitely repeating sequence of exactly one tuple below or a cyclic shift of one tuple below.
\end{minipage}
\begin{minipage}{.6\textwidth}
\begin{figure}[H]
\begin{center}
\modcircleWithBlack{5}{4}{11}{.65}
\end{center}
%\vspace{-.25in}
\caption{ $\Ftensubfill{k}{5}$ diagrams}
\label{fig:r=5}
\end{figure}
\end{minipage}

\bigskip

\begin{center}
\begin{tabular}{|c||c|}
\hline
$\mathcal{D}_0$ & 0 5 5 0 5 5 0 5 5 0 5 5 \\ \hline
$\mathcal{D}_1$ & 1 8 9 7 6 3 9 2 1 3 4 7 \\ \hline
$\mathcal{D}_2$ & 1 3 4 7 1 8 9 7 6 3 9 2 \\ \hline
$\mathcal{D}_3$ & 2 1 3 4 7 1 8 9 7 6 3 9 \\ \hline
$\mathcal{D}_4$ & 3 4 7 1 8 9 7 6 3 9 2 1 \\ \hline
\end{tabular}
\end{center}

\noindent For any $k \in \{0, 1, \ldots, 59\}$, we make the following observations:
\begin{itemize}
\item The sum of the entries in $\mathcal{D}_k$ equals 40 if $k \equiv 0 \pmod{5}$ and equals 60 if $k \not\equiv 0 \pmod{5}$.
\item If $k \equiv 0 \pmod{5}$, then the $\mathcal{D}_k$ tuples contain the numbers 0, 5, 5, repeated in that order or repeated in some cyclic shift.
\item If $k \not\equiv 0 \pmod{5}$, then the $\mathcal{D}_k$ tuples are cyclic shifts of the complete period of the Lucas sequence modulo 10; that is, $\left( L_n \pmod{10} \right)_{n=0}^{11} = ( 2, 1, 3, 4, 7, 1, 8, 9, 7, 6, 3, 9)$. Hence for every $k$ not divisible by 5, the subsequence $\Ftensubfill{k}{5}$ is essentially the Lucas sequence modulo 10.
\end{itemize}
We leave it to the interested reader to formally prove the assertions above.

In particular, the twelve ordered elements of each $\mathcal{D}_k$ form a cycle such that the sum modulo 10 of any two adjacent terms yields the next term. That is, the consecutive terms in the cyclic sequence exhibit a \textit{Fibonacci-esque} recurrence relation. We prove why this holds in more generality in the next section (see Theorem~\ref{thm:forward_quasi}).

%%%%%%%%%%%%%%%%%%%%%%%%%%%%%%%%%%%%%
%%%%%%%%%%%%%%%%%%%%%%%%%%%%%%%%%%%%%
%%%%%%%%%    SECTION 4    %%%%%%%%%%%
%%%%%%%%%%%%%%%%%%%%%%%%%%%%%%%%%%%%%
%%%%%%%%%%%%%%%%%%%%%%%%%%%%%%%%%%%%%

\section{Quasi-Fibonacci subsequences}\label{sec:quasi_fibonacci_subsequence}

In this section, we examine an interesting class of subsequences $\Ftensub$ that exhibit a \textit{Fibonacci-esque} recurrence relation. We make this precise in the following definition.

\begin{definition}
The subsequence $\Ftensub$ is called a \textit{forward quasi-Fibonacci subsequence} if its terms satisfy the recurrence relation $\F{k+r(j-1)} + \F{k+rj} \equiv \F{k+r(j+1)} \pmod{10}$ for all $j \in \mathbb{Z}$. The subsequence $\Ftensub$ is called a \textit{reverse quasi-Fibonacci subsequence} if its terms satisfy the recurrence relation $\F{k+r(j+1)} + \F{k+rj} \equiv \F{k+r(j-1)} \pmod{10}$ for all $j \in \mathbb{Z}$.
\end{definition}

\begin{example}
Consider the subsequence $\Ftensub$ with $k=3$ and jump size $r=25$. By Theorem~\ref{thm:3_types_of_diagrams}, this subsequence is periodic of length $n = \frac{60}{\gcd(25,60)} = \frac{60}{5} = 12$ and has $q$-value of $q = \frac{25}{\gcd(25,60)} = 5$. Hence it turns out that the corresponding subsequence diagram is the Type~2 star polygon $\{\frac{12}{5}\}$. In Appendix~\ref{appendix_A}, we show step-by-step how this diagram is built in 12 consecutive steps. This subsequence repeats the following 12 period terms:
$$ 2, 1, 3, 4, 7, 1, 8, 9, 7, 6, 3, 9, \ldots $$
It is readily verified that  $\F{k+r(j-1)} + \F{k+rj} \equiv \F{k+r(j+1)} \pmod{10}$ holds for all $j \in \mathbb{Z}$, when $k=3$ and $r=25$, and hence this is a forward quasi-Fibonacci subsequence. As we will prove in Theorem~\ref{thm:forward_quasi}, the fact that $\Ftensubfill{3}{25}$ is a forward quasi-Fibonacci subsequence follows since our $r$-value of 25 satisfies the two conditions: $r \equiv 1 \pmod{4}$ and 3 does not divide $r$.
\end{example}

\begin{remark}\label{rem:r_equals_25_gives_Lucas_sequence}
It worthy to note that the 12 terms that comprise the complete period of the subsequence $\Ftensubfill{3}{25}$ given in the previous example are exactly the terms in the period of the Lucas sequence modulo 10. That is, we have
$$ \Ftensubfill{3}{25} = \left( L_n \pmodd{10} \right)_{n=0}^\infty = ( 2, 1, 3, 4, 7, 1, 8, 9, 7, 6, 3, 9, \ldots ). $$
This same phenomena occurs if we change our jump size from $r=25$ to $r=5$, as we saw in Subsection~\ref{subsec:dodecagon}. That is, we have
$$ \Ftensubfill{3}{5} = \left( L_n \pmodd{10} \right)_{n=0}^\infty = ( 2, 1, 3, 4, 7, 1, 8, 9, 7, 6, 3, 9, \ldots ). $$
\end{remark}

To prove the results in this section, we rely on the well-known identities given in Proposition~\ref{prop:fundamental_identities}. From Identities~\eqref{eq:fund_identity_1} and \eqref{eq:fund_identity_4} in particular, we have the next two results used often in the next two subsections and the proofs of our main results of this section, namely Theorems~\ref{thm:forward_quasi} and \ref{thm:reverse_quasi}.

\begin{lemma}\label{lem:odd_even r}
The following identity holds:
$$\F{-n} \equiv \begin{cases} -\F{n} \pmodd{10} &\mbox{if $n$ is even}, \\
\F{n} \pmodd{10} &\mbox{if $n$ is odd}. \end{cases}$$
\end{lemma}

\begin{proof}
Result follows immediately from Identity~\eqref{eq:fund_identity_1} of Proposition~\ref{prop:fundamental_identities}.
\end{proof}

\begin{lemma}\label{lem:Vorobiev}
For all $m,n \in \mathbb{Z}$, the following identity holds:
$$\F{n+m} \equiv \F{n-1} \F{m} + \F{n} \F{m+1} \pmodd{10}.$$
\end{lemma}

\begin{proof}
Result follows immediately from Identity~\eqref{eq:fund_identity_4} of Proposition~\ref{prop:fundamental_identities}.
\end{proof}

%%%%%%%%%%%%%%%%%%%%%%%%%%%
%%%%%%%%%%%%%%%%%%%%%%%%%%%

\subsection{Forward quasi-Fibonacci subsequences}

\begin{lemma}\label{lem:ECDC Lemma 1}
If $r\equiv 1 \pmod{4}$ and $3$ does not divide $r$, then $1+ \F{1-r} \equiv \F{r+1} \pmod{10}$.
\end{lemma}

\begin{proof}
Let $r \in \mathbb{N}$ such that $r\equiv 1 \pmod{4}$ and 3 does not divide $r$. Since $r \equiv 1 \pmod{4}$ implies that $r$ is odd and hence $1-r$ is even, it follows that $\F{1-r} \equiv  -\F{r-1} \pmod{10}$ by Lemma~\ref{lem:odd_even r}. Hence it suffices to show $\F{r-1}+\F{r+1} \equiv 1 \pmod{10}$ since that would imply the following sequence of implications:
\begin{align*}
    \F{r-1}+\F{r+1} \equiv 1 \pmodd{10} &\Longrightarrow -\F{1-r} + \F{r+1} \equiv 1 \pmodd{10}\\
    &\Longrightarrow 1 + \F{1-r} \equiv \F{r+1} \pmodd{10}.
\end{align*}
By assumption we know 3 does not divide $r$,  so exactly one of $r-1$ or $r+1$ is divisible by 3. Therefore by Identity~\eqref{eq:fund_identity_2a}, exactly one of $F_{r-1}$ or $F_{r+1}$ is even and consequently $\F{r-1}+\F{r+1}\pmod{10}$ is odd. Thus this sum modulo 10 must be congruent to an element of the set $A=\{1,3,5,7,9\}$. 
Since $F_{r-1}+F_{r+1}=L_r$ by Identity~\eqref{eq:fund_identity_3}, it follows that $\F{r-1}+\F{r+1}\equiv L_r \pmod{10}$.
Observe that the Lucas sequence modulo 5 has a period of length 4 as follows:
$$\bigl(L_n \pmodd{5}\bigr)_{n=0}^\infty = (2,1,3,4,2,1,3,4,\ldots).$$
So for $r \equiv 1 \pmod{4}$, we have $L_r \equiv 1 \pmod{5}$ and thus $L_r$ modulo 10 must be congruent to an element of the set $B = \{1,6\}$. Since $\F{r-1}+\F{r+1}\equiv L_r \pmod{10}$, then $\F{r-1}+\F{r+1}$ modulo 10 lies in $A \cap B = \{1\}$. Therefore $\F{r-1}+\F{r+1} \equiv 1 \pmod{10}$ and the result follows.
\end{proof}

\begin{theorem}\label{thm:forward_quasi}
If $r \equiv 1 \pmod{4}$ and $3$ does not divide $r$, then $\Ftensub$ is a forward quasi-Fibonacci subsequence for all $k \geq 0$. More precisely for all $j \in \mathbb{Z}$, we have
$$\F{k+r(j-1)} + \F{k+rj} \equiv \F{k+r(j+1)} \pmodd{10}.$$
\end{theorem}

\begin{proof}
Fix $k \in \mathbb{N}\cup\{0\}$, and let $r \in \mathbb{N}$ such that $r\equiv 1 \pmod{4}$ and 3 does not divide $r$. Observe the sequence of equalities and congruences

\begin{align*}
    \F{k+r(j-1)} &+ \F{k+rj}\\
    &= \F{k+rj} + \F{(k+rj)-r}\\
    &\equiv  \F{k+rj} + \bigl( \F{(k+rj)-1} \F{-r}+ \F{k+rj} \F{-r+1} \bigr) \pmodd{10} &\mbox{by Lemma~\ref{lem:Vorobiev}}\\
    &\equiv \F{(k+rj)-1} \F{-r} +\F{k+rj} (1+\F{1-r}) \pmodd{10}\\
    &\equiv \F{(k+rj)-1} \F{-r} + \F{k+rj} \F{r+1} \pmodd{10}
    &\mbox{by Lemma~\ref{lem:ECDC Lemma 1}}\\
    &\hspace{-.15in}^{(\textcolor{red}{\ast})}\equiv \F{(k+rj)-1} \F{r} + \F{k+rj} \F{r+1} \pmodd{10} &\mbox{by Lemma~\ref{lem:odd_even r}}\\
    &\equiv \F{(k+rj)+r} \pmodd{10} &\mbox{by Lemma~\ref{lem:Vorobiev}}\\
    &\equiv \F{k+r(j+1)} \pmod{10},
\end{align*}
where congruence (\textcolor{red}{$\ast$}) holds since $r \equiv 1 \pmod{4}$ implies $r$ is odd. Hence the result follows.
\end{proof}

%%%%%%%%%%%%%%%%%%%%%%%%%%%
%%%%%%%%%%%%%%%%%%%%%%%%%%%

\subsection{Reverse quasi-Fibonacci subsequences}

\begin{lemma}\label{lem:ECDC Lemma 1a}
If $r\equiv 3 \pmod{4}$ and $3$ does not divide $r$, then $1 + \F{r+1} \equiv \F{1-r} \pmod{10}$.
\end{lemma}

\begin{proof}
Let $r \in \mathbb{N}$ such that $r\equiv 3 \pmod{4}$ and 3 does not divide $r$. Since $r \equiv 3 \pmod{4}$ implies that $r$ is odd and hence $1-r$ is even, it follows that $\F{1-r} \equiv  -\F{r-1} \pmod{10}$ by Lemma~\ref{lem:odd_even r}. Hence it suffices to show $\F{r-1}+\F{r+1} \equiv 9 \pmod{10}$ since that would imply the following sequence of implications:
\begin{align*}
    \F{r-1}+\F{r+1} \equiv 9 \pmodd{10} &\Longrightarrow -\bigl(\F{r-1} + \F{r+1} \bigr) \equiv 1 \pmodd{10}\\
    &\Longrightarrow \F{1-r} - \F{r+1} \equiv 1 \pmodd{10}\\
    &\Longrightarrow 1 + \F{r+1} \equiv \F{1-r} \pmodd{10}.
\end{align*}
As in the proof of Lemma~\ref{lem:ECDC Lemma 1}, the sum $\F{r-1}+\F{r+1}$ modulo 10 must be congruent to an element of the set $A=\{1,3,5,7,9\}$, and moreover, $\F{r-1}+\F{r+1}\equiv L_r \pmod{10}$. However this time since $r \equiv 3 \pmod{4}$, we have $L_r \equiv 4 \pmod{5}$ and thus $L_r$ modulo 10 must be congruent to an element of the set $B = \{4,9\}$. Since $\F{r-1}+\F{r+1}\equiv L_r \pmod{10}$, then $\F{r-1}+\F{r+1}$ modulo 10 lies in $A \cap B = \{9\}$. Therefore $\F{r-1}+\F{r+1} \equiv 9 \pmod{10}$ and the result follows.
\end{proof}

\begin{theorem}\label{thm:reverse_quasi}
If $r \equiv 3 \pmod{4}$ and $3$ does not divide $r$, then $\Ftensub$ is a reverse quasi-Fibonacci subsequence for all $k \geq 0$. More precisely for all $j \in \mathbb{Z}$, we have
$$\F{k+r(j+1)} + \F{k+rj} \equiv \F{k+r(j-1)} \pmodd{10}.$$
\end{theorem}

\begin{proof}
Fix $k \in \mathbb{N}\cup\{0\}$, and let $r \in \mathbb{N}$ such that $r\equiv 3 \pmod{4}$ and 3 does not divide $r$. Observe the sequence of equalities and congruences
\begin{align*}
    \F{k+r(j+1)} &+ \F{k+rj}\\
    &= \F{k+rj} + \F{(k+rj)+r}\\
    &\equiv  \F{k+rj} + \bigl( \F{(k+rj)-1} \F{r}+ \F{k+rj} \F{r+1} \bigr) \pmodd{10} &\mbox{by Lemma~\ref{lem:Vorobiev}}\\
    &\equiv \F{(k+rj)-1} \F{r} +\F{k+rj} (1+\F{r+1}) \pmodd{10}\\
    &\equiv \F{(k+rj)-1} \F{r} + \F{k+rj} \F{1-r} \pmodd{10}
    &\mbox{by Lemma~\ref{lem:ECDC Lemma 1a}}\\
    &\hspace{-.15in}^{(\textcolor{red}{\ast})}\equiv \F{(k+rj)-1} \F{-r} + \F{k+rj} \F{1-r} \pmodd{10} &\mbox{by Lemma~\ref{lem:odd_even r}}\\
    &\equiv \F{(k+rj)-1} \F{-r} + \F{k+rj} \F{-r+1} \pmodd{10}\\
    &\equiv \F{(k+rj)-r} \pmodd{10} &\mbox{by Lemma~\ref{lem:Vorobiev}}\\
    &\equiv \F{k+r(j-1)} \pmod{10},
\end{align*}
where congruence (\textcolor{red}{$\ast$}) holds since $r \equiv 3 \pmod{4}$ implies $r$ is odd. Hence the result follows.
\end{proof}

%%%%%%%%%%%%%%%%%%%%%%%%%%%%%%%%%%%%%
%%%%%%%%%%%%%%%%%%%%%%%%%%%%%%%%%%%%%
%%%%%%%%%    SECTION 5    %%%%%%%%%%%
%%%%%%%%%%%%%%%%%%%%%%%%%%%%%%%%%%%%%
%%%%%%%%%%%%%%%%%%%%%%%%%%%%%%%%%%%%%

\section{Complete Fibonacci subsequences}\label{sec:complete_fibonacci_subsequences}

In Remark~\ref{rem:r_equals_25_gives_Lucas_sequence}, we noted that certain subsequences $\Ftensub$ exhibited periods that were exactly the complete length 12 Lucas sequence modulo 10. In this section, we unveil the most tantalizing of the subsequences $\Ftensub$, namely ones that yield periods which are the complete length 60 Fibonacci sequence modulo 10. We make this more precise in the definition below.

\begin{definition}
The subsequence $\Ftensub$ is called a \textit{forward (respectively, reverse) complete Fibonacci subsequence} if it is a forward (respectively, reverse) quasi-Fibonacci sequence with a Type~3 subsequence diagram, in which case $\Ftensub$ coincides with $\FtensubNkrFwd$ (respectively, $\FtensubNkrRev$) for some $N_{k,r} \in \{0, 1, \ldots, 59\}$.
\end{definition}

\begin{example}\label{example:complete_subsequence}
Consider the subsequence $\Ftensub$ with $k=9$ and $r=13$. By Theorem~\ref{thm:3_types_of_diagrams}, this subsequence is periodic of length $n = \frac{60}{\gcd(13,60)} = 60$ and has $q = \frac{13}{\gcd(13,60)} = 13$. Hence the corresponding diagram is the Type~3 star polygon $\{\frac{60}{13}\}$.

In Figure~\ref{fig:r=13_few_edges}, we show the first ten edges of this diagram starting at the first subsequence term, vertex $\F{9}$, and proceeding clockwise jumping every $13$ points. Observe that the 60 terms of the first period $\left(\newF_{10,9+13j}\right)_{j=0}^{59}$ arising from the vertex labels in the subsequence diagram are as follows:
\begin{align*}
&4, 1, 5, 6, 1, 7, 8, 5, 3, 8, 1, 9, 0, 9, 9, 8, 7, 5, 2, 7, 9, 6, 5, 1, 6, 7, 3, 0, 3, 3,\\
&6, 9, 5, 4, 9, 3, 2, 5, 7, 2, 9, 1, \red{0, 1, 1, 2, 3, 5, 8, 3, 1, 4, 5, 9, 4, 3, 7, 0, 7, 7}.
\end{align*}
Surprisingly, we recover the original Fibonacci sequence modulo 10 in its forward direction; more precisely, the subsequence $\Ftensubfill{9}{13}$ is exactly $\FtensubNkrFwd$ for some $N_{k,r} \in \{0, 1, \ldots, 59\}$, and hence $\Ftensubfill{9}{13}$ is a forward complete Fibonacci subsequence.

In red, we highlight the last 18 elements in the period of the subsequence. In these last 18 elements we can see the original parent sequence $\Ften$ beginning. Observe that in the subsequence $\Ftensubfill{9}{13}$, this begins at the index $j=60 - 18 = 42$. Indeed, knowing that the $42^{\mathrm{nd}}$ term in the subsequence is precisely where the original Fibonacci sequence modulo 10 begins is crucial in finding the value $N_{9,13}$, which in this case happens to be the number 18 itself; that is, $\Ftensubfill{9}{13}$ is exactly the sequence $\left(\F{n}\right)_{n=N_{9,13}}^\infty$ where $N_{9,13} = 18$. In Subsection~\ref{subsec:computing_j_naught_and_N_k_r}, we show how to explicitly compute this $N_{k,r}$ value for any complete Fibonacci subsequence $\Ftensub$ via the method given in Algorithm~\ref{algorithm:find_magic_carrot}.
\end{example}

\begin{figure}[H]
\begin{center}
\modcirclenk{13}{9}{9}{0}{9}{.7}
\end{center}
%\vspace{-.25in}
\caption{First ten edges of the $\Ftensubfill{9}{13}$ diagram}
\label{fig:r=13_few_edges}
\end{figure}

%%%%%%%%%%%%%%%%%%%%%%%%%%%
%%%%%%%%%%%%%%%%%%%%%%%%%%%

\subsection{Group theoretic preliminaries}\label{subsec:group_theory}

Before we give the main results of this section, we first recall some necessary group theory relevant to our Fibonacci setting.

\begin{definition}
The \textit{multiplicative group of units modulo} $m$ is the set of congruence classes in the ring $\mathbb{Z}/m\mathbb{Z}$ represented by integers coprime to $m$. Taking the the least residue class representatives and denoting this set of units by $U(m)$, we have $U(m) := \{ 1 \leq r \leq m-1 \mid \gcd(r,m) = 1 \}$.
The size of $U(m)$ is $\phi(m)$ where $\phi$ is the Euler phi function.
\end{definition}

In this section, we are concerned primarily with the groups $U(10)$ and $U(60)$. The group $U(10)$ has size $\phi(10)=4$, and its elements and corresponding inverses are as follows:
\begin{table}[H]
\begin{center}
\renewcommand{\arraystretch}{1.1}
\begin{tabular}{|c||c|c|c|c|}
    \hline
    $r \in U(10)$ & 1 & 3 & 7 & 9 \\ \hline
    $r^{-1} \in U(10)$ & 1 & 7 & 3 & 9 \\ \hline
\end{tabular}
\caption{The group $U(10)$ and its inverses}
\label{table:group_U_10}
\end{center}
\end{table}

\noindent The group $U(60)$ has size $\phi(60)=16$, and its elements and corresponding inverses are as follows:
\begin{table}[H]
\begin{center}
\renewcommand{\arraystretch}{1.1}
\begin{tabular}{|c||c|c|c|c|c|c|c|c|c|c|c|c|c|c|c|c|}
    \hline
    $r \in U(60)$ & 1 & 7 & 11 & 13 & 17 & 19 & 23 & 29 & 31 & 37 & 41 & 43 & 47 & 49 & 53 & 59 \\ \hline
    $r^{-1} \in U(60)$ & 1 & 43 & 11 & 37 & 53 & 19 & 47 & 29 & 31 & 13 & 41 & 7 & 23 & 49 & 17 & 59 \\ \hline
\end{tabular}
\caption{The group $U(60)$ and its inverses}
\label{table:group_U_60}
\end{center}
\end{table}

By examining Table~\ref{table:group_U_60}, we observe a remarkable connection between each subscript $r$ and the value $\F{r}$, whenever $r$ lies in $U(60)$. We state this tantalizing result in the following lemma.

\begin{lemma}\label{lem:tantalizing_congruence_result}
For all $r \in U(60)$,
\[
 \F{r} \equiv 
  \begin{cases} 
   r \pmodd{10} & \text{if } r \equiv 1 \pmodd{4},\\
   -r \pmodd{10} & \text{if } r \equiv 3 \pmodd{4}.
  \end{cases}
\]
\end{lemma}

\begin{proof}
By brute force verification of the sixteen $\F{r}$ values in Table~\ref{table:Fib_values_when_r_in_U_60}, this result is confirmed.
\end{proof}

\begin{table}[H]
\begin{center}
\renewcommand{\arraystretch}{1.1}
\begin{tabular}{|c||c|c|c|c|c|c|c|c|c|c|c|c|c|c|c|c|}
    \hline
    $r \in U(60)$ & \bluebf{1} & \redbf{7} & \redbf{11} & \bluebf{13} & \bluebf{17} & \redbf{19} & \redbf{23} & 29 & \redbf{31} & \bluebf{37} & \bluebf{41} & \redbf{43} & \redbf{47} & \bluebf{49} & \bluebf{53} & \redbf{59} \\ \hline\hline
    $r \pmod{10}$ & 1 & 7 & 1 & 3 & 7 & 9 & 3 & 9 & 1 & 7 & 1 & 3 & 7 & 9 & 3 & 9 \\ \hline
    $\F{r}$ & 1 & 3 & 9 & 3 & 7 & 1 & 7 & 9 & 9 & 7 & 1 & 7 & 3 & 9 & 3 & 1 \\ \hline
\end{tabular}
\caption{$\F{r}$ values when $r \in U(60)$ with $r \equiv 1\pmod{4}$ in blue and $r \equiv 3\pmod{4}$ in red}
\label{table:Fib_values_when_r_in_U_60}
\end{center}
\end{table}

The next lemma will be helpful when we prove the main results of this section. Like the previous lemma, this lemma can be verified by brute force verification of each of the sixteen $\F{r}$ values; however, we provide a more intuitive proof that explains why this intriguing statement is true.

\begin{lemma}\label{lem:r_in_U_60_then_Fib_val_is_in_U_10}
If $r \in U(60)$, then $\F{r} \in U(10)$.
\end{lemma}

\begin{proof}
Let $r \in U(60)$. Then certainly neither 3 nor 5 divides $r$. Identities~\eqref{eq:fund_identity_2a} and \eqref{eq:fund_identity_2b}, respectively, of Proposition~\ref{prop:fundamental_identities} imply the following two equivalencies:
\begin{align*}
    \mbox{3 divides $r$} &\Longleftrightarrow \mbox{2 divides $\F{r}$}\\
    \mbox{5 divides $r$} &\Longleftrightarrow \mbox{5 divides $\F{r}$}.
\end{align*}
Therefore if $r \in U(60)$, then neither 2 nor 5 can divide $\F{r}$. Thus $\F{r}$ must be one of 1, 3, 7, or 9. Hence $\F{r} \in U(10)$ as desired.
\end{proof}

The fact that $U(10)$ is a multiplicative group, and hence closed under inverses, allows one to prove the following useful lemma.

\begin{lemma}\label{lem:r_in_U_60_then_Fib_val_has_an_inverse}
Let $r \in U(60)$. Then $\F{M} \F{r} \equiv 1 \pmod{10}$ for some $M \in \mathbb{Z}$. In particular, these specific choices of $M$ work:
\[
M =
\begin{cases}
    0 \pm 1 & \text{if } \F{r} = 1,\\
    15 \pm 1 & \text{if } \F{r} = 3,\\
    45 \pm 1 & \text{if } \F{r} = 7,\\
    30 \pm 1 & \text{if } \F{r} =9.
\end{cases}
\]
\end{lemma}

\begin{proof}
Let $r \in U(60)$. By Lemma~\ref{lem:r_in_U_60_then_Fib_val_is_in_U_10}, we know $\F{r} \in \{1,3,7,9\}$. This set is the group $U(10)$, and hence every element has a multiplicative inverse. Therefore, many choices of $M \in \mathbb{Z}$ exists such that $\F{M} \F{r} \equiv 1 \pmod{10}$ since the values 1, 3, 7, and 9 all appear in the $\Ften$ sequence. To find the particular choices of $M$ given in the lemma statement, set $M_i := 15i$ and $r_i := 3^i \pmod{10}$ for $i \in \{0,1,2,3\}$. Then we get the following table of values:
\begin{center}
    \begin{tabular}{|c||c|c|c|c|}
    \hline
        $i$ & 0 & 1 & 2 & 3 \\ \hline \hline
        $M_i$ & 0 & 15 & 30 & 45 \\ \hline
        $r_i$ & 1 & 3 & 9 & 7 \\ \hline
        $\F{M_i \pm 1}$ & 1 & 7 & 9 & 3\\ \hline
    \end{tabular}
\end{center}
One easily verifies for each $i \in \{0,1,2,3\}$ that if $\F{r} = r_i$, then $\F{M_i \pm 1} \F{r} \equiv 1 \pmod{10}$, and the result follows.
\end{proof}

%%%%%%%%%%%%%%%%%%%%%%%%%%%
%%%%%%%%%%%%%%%%%%%%%%%%%%%

\subsection{Forward and reverse complete Fibonacci subsequences}\label{subsec:forward_and_reverse_complete_Fib_subsequence}

To prove the main results of this subsection, we need the next two lemmas. This first one analyzes the arithmetic progression of subscripts in the subsequence $\Ftensub$ to provide an argument why there are exactly four equally spaced zeros in the length 60 period of the subsequence $\Ftensub$ when $r \in U(60)$.

\begin{lemma}\label{lem:four_zeros}
If $r \in U(60)$, then the arithmetic progression sequence $\left( k + rj \right)_{j=0}^{59}$ contains exactly four terms divisible by $15$. Moreover, these four terms are equally spaced apart in consecutive jumps of size $15$. In particular, each period of $\Ftensub$ contains exactly four equally spaced zeros.
\end{lemma}

\begin{proof}
Let $r \in U(60)$ and consider $\left( k + rj \right)_{j=0}^{59}$. We will show that 15 divides the term $k + r j_0$ for some $j_0 \in \{0, 1, \ldots 14\}$, and that the remaining three terms divisible by 15 are $k + r (j_0 + 15i)$ for $i = 1,2,3$. We claim the first 15 terms give distinct residue classes modulo 15. To that end, suppose $k + r a \equiv k + r b \pmod{15}$ for some $0 \leq a \leq b \leq 14$ and hence we have $r a \equiv r b \pmod{15}$. But $r \in U(60)$ implies $\gcd(r,15)=1$ and so $a \equiv b \pmod{15}$. So by the pigeonhole principle, exactly one of the first 15 terms is congruent to 0 modulo 15; that is, 15 divides the term $k + r j_0$ for exactly one value $j_0 \in \{0, 1, \ldots 14\}$. Now suppose 15 also divides $k + r (j_0 + t)$ for some $t \geq 1$. Then 15 divides the difference $[k + r (j_0 + t)] - [k + r j_0]$, which is $rt$. But $r \in U(60)$ implies  15 does not divide $r$, and hence 15 must divide $t$. Thus the four terms in $\left( k + rj \right)_{j=0}^{59}$ divisible by 15 are $k + r (j_0 + 15i)$ for $i = 0,1,2,3$.

Consequently by Lemma~\ref{lem:Pisano_divis_by_5}, the terms $\F{k + r (j_0 + 15i)}$ for $i = 0,1,2,3$ are the only terms in the first period of the subsequence $\Ftensub$ equal to 0, and in particular, each period contains exactly four equally spaced zeros.
\end{proof}

\begin{corollary}\label{cor:zeros_of_subsequence_are_0_15_30_and_45}
If $r \in U(60)$, then in the arithmetic progression sequence $\left( k + rj \right)_{j=0}^{59}$ the four terms divisible by $15$ are congruent modulo $60$ to the four values $0, 15, 30, 45$, in some order.
\end{corollary}

\begin{proof}
When $r \in U(60)$, the subsequence diagram corresponding to $\Ftensub$ is Type 3 by Theorem~\ref{thm:3_types_of_diagrams}, and hence the period $\FtensubCompletePeriod$ uses every vertex in the $\Ften$-circle. In particular, the subsequence diagram for $\Ftensub$ includes vertices $\F{15i}$ for each $i \in \{0,1,2,3\}$. Hence the four subscripts $k + r (j_0 + 15i)$ for $i = 0,1,2,3$ divisible by 15 are congruent modulo 60 to the four values in the set $\{0, 15, 30, 45\}$.
\end{proof}

\begin{lemma}\label{lem:zero_followed_by_a_one}
If $r \in U(60)$, then there exists a $0$ followed by a $1$ in the period $\FtensubCompletePeriod$ of the subsequence $\Ftensub$.
\end{lemma}

\begin{proof}
By Lemma~\ref{lem:four_zeros}, for $r \in U(60)$ the period $\FtensubCompletePeriod$ of $\Ftensub$ contains four equally spaced zeros. Moreover, these four zeros occur at $\F{k + r (j_0 + 15i)}$ for $i = 0, 1, 2, 3$, where $\F{k + r j_0}$ is the first zero of the period for some $j_0 \in \{0, 1, \ldots, 14\}$. For ease of notation, set $h_i := j_0 + 15i$. Then $\F{k + r h_i} = 0$ and 15 divides $k + r h_i$ for all $i \in \mathbb{Z}$. Observe that
\begin{align*}
    \F{k+r(h_i + 1)} &= \F{k + r h_i + r}\\
    &\equiv \F{(k + r h_i) - 1} \F{r} + \F{k + r h_i} \F{r+1} \pmodd{10} &\mbox{by Lemma~\ref{lem:Vorobiev}}\\
    &\equiv \F{(k + r h_i) - 1} \F{r} \pmodd{10},
\end{align*}
where the second congruence holds since $\F{k + r h_i} = 0$ in the first congruence. It suffices to show that we can choose an $i \in \{0,1,2,3\}$ such that $\F{(k + r h_i) - 1}$ is the multiplicative inverse of $\F{r}$ in the group $U(10)$. However, that is precisely what Lemma~\ref{lem:r_in_U_60_then_Fib_val_has_an_inverse} states. Observe that $\F{0 \pm 1} = 1$, $\F{15 \pm 1} = 7$, $\F{30 \pm 1} = 9$, and $\F{45 \pm 1} = 3$, so regardless of what value $\F{r}$ is in the set $\{1,3,7,9\}$, Lemma~\ref{lem:r_in_U_60_then_Fib_val_has_an_inverse} guarantees the existence of an inverse of the form $\F{15i \pm 1}$ for some $i \in \{0,1,2,3\}$. As a consequence of Corollary~\ref{cor:zeros_of_subsequence_are_0_15_30_and_45}, the sets $\{ \F{(k + r h_i) - 1} \}_{i=0}^3$ and $\{ \F{15i - 1} \}_{i=0}^3$ coincide. Hence we can choose an $i \in \{0,1,2,3\}$ such that $\F{k+r(h_i + 1)} \equiv \F{(k + r h_i) - 1} \F{r} \equiv 1 \pmod{10}$, and thus for this $i$ we have $\F{k+rh_i} = 0$ and $\F{k+r(h_i + 1)} = 1$ so there is a 0 followed by a 1 in the period $\FtensubCompletePeriod$.
\end{proof}

We are now ready to prove the main results of this subsection in Theorems~\ref{thm:forward_complete_Fibonacci_result} and \ref{thm:reverse_complete_Fibonacci_result}. Though the two proofs look very similar, there are enough important subtle differences to give thorough proofs for each.

\begin{theorem}\label{thm:forward_complete_Fibonacci_result}
Let $r \in U(60)$ such that $r \equiv 1 \pmod{4}$. Then $\Ftensub$ is a forward complete Fibonacci subsequence.
\end{theorem}

\begin{proof}
Given that $r \in U(60)$ with $r \equiv 1 \pmod{4}$, the results of Theorem~\ref{thm:3_types_of_diagrams} and \ref{thm:forward_quasi} together imply that $\Ftensub$ is a forward quasi-Fibonacci subsequence with a Type 3 diagram. That is, the subsequence satisfies the recurrence $\F{k+r(j-1)} + \F{k+rj} \equiv \F{k+r(j+1)} \pmod{10}$ for all $j \in \mathbb{Z}$ and is periodic of length 60. Moreover, by Lemma~\ref{lem:zero_followed_by_a_one}, we know that this length 60 period $\FtensubCompletePeriod$ contains a 0 followed by a 1, and hence the subsequence $\Ftensub$ coincides with the original parent sequence $\FtensubNkrFwd$ starting at some index $N_{k,r} \in \{0, 1, \ldots, 59\}$ dependent on the values $k$ and $r$. We conclude that $\Ftensub$ is forward complete Fibonacci subsequence.
\end{proof}

\begin{theorem}\label{thm:reverse_complete_Fibonacci_result}
Let $r \in U(60)$ such that $r \equiv 3 \pmod{4}$. Then $\Ftensub$ is a reverse complete Fibonacci subsequence.
\end{theorem}

\begin{proof}
Given that $r \in U(60)$ with $r \equiv 3 \pmod{4}$, the results of Theorem~\ref{thm:3_types_of_diagrams} and \ref{thm:reverse_quasi} together imply that $\Ftensub$ is a reverse quasi-Fibonacci subsequence with a Type 3 diagram. That is, the subsequence satisfies the recurrence $\F{k+r(j+1)} + \F{k+rj} \equiv \F{k+r(j-1)} \pmod{10}$ for all $j \in \mathbb{Z}$ and is periodic of length 60. Moreover, by Lemma~\ref{lem:zero_followed_by_a_one}, we know that this length 60 period $\FtensubCompletePeriod$ contains a 0 followed by a 1, and hence the subsequence $\Ftensub$ coincides with the original parent sequence but running in reverse, namely the sequence $\FtensubNkrRev$, starting at some index $N_{k,r} \in \{0, 1, \ldots, 59\}$ dependent on the values $k$ and $r$. We conclude that $\Ftensub$ is reverse complete Fibonacci subsequence.
\end{proof}

%%%%%%%%%%%%%%%%%%%%%%%%%%%
%%%%%%%%%%%%%%%%%%%%%%%%%%%

\subsection{Methods to compute the values \texorpdfstring{$j_0$}{j0} and \texorpdfstring{$N_{k,r}$}{Nkr}}\label{subsec:computing_j_naught_and_N_k_r}

In this subsection, we give a formula and algorithm, respectively, to compute the values $j_0$ and $N_{k,r}$. In particular, we do the following:
\begin{enumerate}
    \item Compute the first zero of $\Ftensub$; that is, the index value $j_0$ such that $\F{k + r j_0} = 0$ and $j_0 \in \{0, 1, \ldots, 14\}$.
    \item Compute the value $N_{k,r} \in \{0, 1, \ldots, 59\}$ such that $\Ftensub$ coincides with $\FtensubNkrFwd$ (respectively, $\FtensubNkrRev$) when $\Ftensub$ is a forward (respectively, reverse) complete Fibonacci subsequence.
\end{enumerate}

\begin{formula}\label{formula:finding_first_zero}
Let $r \in U(60)$. The first zero in the complete Fibonacci subsequence $\Ftensub$ occurs at the index value $j_0 \in \{0, 1, \ldots, 14\}$ satisfying the congruence $j_0 \equiv r^{-1}(-k) \pmod{15}$.
\end{formula}

\begin{proof}
Let $r \in U(60)$ and $k \in \{0, 1, \ldots 59\}$. We know that $r^{-1}$ exists since $r \in U(60)$ and $U(60)$ is a group. Choose $j_0 \in \{0, 1, \ldots, 14\}$ such that $j_0 \equiv r^{-1}(-k) \pmod{15}$ holds. Then we have the sequence of implications
\begin{align*}
    j_0 \equiv r^{-1}(-k) \pmodd{15} &\;\Longrightarrow\; r j_0 \equiv r r^{-1}(-k) \pmodd{15}\\
    &\;\Longrightarrow\; r j_0 \equiv -k \pmodd{15}\\
    &\;\Longrightarrow\; k + r j_0 \equiv 0 \pmodd{15},
\end{align*}
and so 15 divides $k + r j_0$. By Lemma~\ref{lem:Pisano_divis_by_5}, it follows that $\F{k + r j_0} = 0$. Observe that since the zeros of $\Ftensub$ are equally spaced in gaps of size 15 by Lemma~\ref{lem:four_zeros}, then there cannot exist a $j$ such that $0 \leq j < j_0$ and $\F{k + r j} = 0$. Therefore, the first zero occurs at the index value $j_0$.
\end{proof}

To compute the $N_{k,r}$ value, it is helpful to recall the definition of a primitive root and the discrete logarithm, namely the index function.

\begin{definition}
A \textit{primitive root of $n$} is an element $g$ in the group of units $U(n)$ such that for each element $u \in U(n)$, there exists an $i \in \mathbb{Z}$ for which $u \equiv g^i \pmod{n}$.
\end{definition}

\begin{definition}
Let $g$ be a primitive root of $n$. If $u \in U(n)$, then the smallest integer $i \in \mathbb{N}$ such that $u \equiv g^i \pmod{n}$ is called the \textit{index of $u$ relative to $g$}. We denote this $i$ value as $\ind_g(u)$.
\end{definition}

We care primarily about the group $U(10)$. This group has two primitive roots, namely the values 3 and 7. We exploit the fact that 3, in particular, is a primitive root of 10 in the following method to compute $N_{k,r}$.

\begin{algorithm}\label{algorithm:find_magic_carrot}
Let $r \in U(60)$. To compute $N_{k,r}$, proceed as follows:
\begin{enumerate}[label=\protect\circled{$\arabic*$}]
    \item Find $\tilde{r} \in \{1,3,7,9\}$ such that $\tilde{r} \equiv 
  \begin{cases} 
   r \pmod{10} & \text{if } r \equiv 1 \pmod{4},\\
   -r \pmod{10} & \text{if } r \equiv 3 \pmod{4}.
  \end{cases}$
    \item Find $i_0 \in \{0,1,2,3\}$ such that $i_0 = \ind_3(\tilde{r})$.
    \item Set $M_{i_0} := 15 i_0$ and find $\tilde{j} \in \{0, 1, \ldots, 59\}$ such that $\tilde{j} \equiv r^{-1} (M_{i_0} - k) \pmod{60}$.
    \item Set $N_{k,r} := \begin{cases} 
   60 - \tilde{j} & \text{if } r \equiv 1 \pmod{4},\\
   \tilde{j} & \text{if } r \equiv 3 \pmod{4}.
  \end{cases}$
\end{enumerate}
We conclude that $\Ftensub$ coincides with $\FtensubNkrFwd$ when $r \equiv 1 \pmod{4}$, and $\Ftensub$ coincides with $\FtensubNkrRev$ when $r \equiv 1 \pmod{3}$.
\end{algorithm}

\begin{proof}
Let $r \in U(60)$ and $k \in \{0, 1, \ldots 59\}$. Observe that the sequence $\bigl(k + rj \pmod{60}\bigr)_{j=0}^\infty$ is just a cyclic shift of the sequence $\bigl(rj \pmod{60}\bigr)_{j=0}^\infty$. Hence if we know where the first $0$ followed by a $1$ occurs in $\FtensubZeroK$, then we can find the index value $\tilde{j} \in \{0, 1, \ldots, 59\}$ where this occurs in $\Ftensub$. More precisely, given $j \in \{0, 1, \ldots, 59\}$ such that $\F{rj} = 0$ and $\F{r(j+1)} = 1$, then in particular we know that $rj \pmod{60}$ is congruent modulo 60 to exactly one multiple of 15 in the set $\{0, 15, 30, 45\}$. We claim this is the value $M_{i_0}$ from step~3 in the algorithm, and hence it suffices to find the value $\tilde{j} \in \{0, 1, \ldots, 59\}$ such that $k + r \tilde{j} \equiv M_{i_0} \pmod{60}$ since that would imply $\F{k + r\tilde{j}} = 0$ and $\F{k + r(\tilde{j} + 1)} = 1$, as desired.

We now show that if $\F{rj} = 0$ and $rj \equiv M_{i_0} \pmod{60}$, then $\F{r(j+1)} = 1$ is forced. Since the congruence $rj \equiv M_{i_0} \pmod{60}$ holds, then $r(j+1) = rj + r \equiv M_{i_0} + r \pmod{60}$. Thus we have
\begin{align*}
    \F{r(j+1)} &= \F{M_{i_0}+r}\\
    &\equiv \F{M_{i_0} - 1} \F{r} + \F{M_{i_0}} \F{r+1} \pmodd{10} & \mbox{by Lemma~\ref{lem:Vorobiev}}\\
    & \equiv \F{M_{i_0} - 1} \F{r} \pmodd{10} &\mbox{since $\F{M_{i_0}} = 0$ by Lemma~\ref{lem:Pisano_divis_by_5}}\\
    &\equiv 1 \pmodd{10},
\end{align*}
where this last congruence holds since our choice of $i_0$ in step~2 of the algorithm is the precise value of $i$ in Lemma~\ref{lem:r_in_U_60_then_Fib_val_has_an_inverse} such that $\F{M_i \pm 1}$ is the multiplicative inverse of $\F{r}$.

It suffices to show that for $\tilde{j} \in \{0, 1, \ldots, 59\}$ with $\tilde{j} \equiv r^{-1} (M_{i_0} - k) \pmod{60}$ as given in step~3 of the algorithm, we have $\F{k + r\tilde{j}} = 0$ and $\F{k + r(\tilde{j} + 1)} = 1$. To show $\F{k + r\tilde{j}} = 0$, observe that
\begin{align*}
    \tilde{j} \equiv r^{-1} (M_{i_0} - k) \pmodd{60} &\;\Longrightarrow\; r \tilde{j} \equiv M_{i_0} - k \pmodd{60}\\
    &\;\Longrightarrow\; k + r \tilde{j} \equiv M_{i_0} \pmodd{60},
\end{align*}
and hence $k + r \tilde{j}$ is a multiple of 15 and so $\F{k + r \tilde{j}} = 0$ by Lemma~\ref{lem:Pisano_divis_by_5}. To show $\F{k + r(\tilde{j} + 1)} = 1$, observe the following sequence of equalities and congruences:
\begin{align*}
    \F{k + r(\tilde{j} + 1)} &= \F{(k + r\tilde{j}) + r}\\
    &= \F{M_{i_0} + r} &\mbox{since $k + r \tilde{j} \equiv M_{i_0} \pmodd{60}$}\\
    &\equiv 1 \pmodd{10},
\end{align*}
where again this last congruence holds for the same reason given in the argument in the previous paragraph. Thus beginning at the term $\F{k+r \tilde{j}}$ the subsequence $\Ftensub$ moves forward (respectively, reverse) through the original parent sequence $\Ften$ if $r \equiv 1 \pmod{4}$ (respectively, $r \equiv 3 \pmod{4}$). Now to compute $N_{k,r}$ it is simply a matter of finding which term in the first Pisano period $\left(\newF_{10,n}\right)_{n=0}^{59}$ of $\Ften$ is exactly the first subsequence term $\F{k+r(0)}$ of $\Ftensub$. There are two cases.

\medskip

\noindent \textbf{Case 1}: If $r \equiv 1 \pmod{4}$, then $\Ftensub$ is a forward complete Fibonacci subsequence with the 0 followed by a 1 beginning at the index value $\tilde{j} \in \{0, 1, \ldots, 59\}$ of $\Ftensub$. Hence the first $\tilde{j}$ terms of the period $\FtensubCompletePeriod$ are exactly the last $\tilde{j}$ terms of the parent sequence's period $\left(\newF_{10,n}\right)_{n = 0}^{59}$, namely the terms $\left(\newF_{10,n}\right)_{n = 60-\tilde{j}}^{59}$. Thus the first term of $\FtensubCompletePeriod$ is the element $\F{60-\tilde{j}}$. We conclude that $\Ftensub = \FtensubNkrFwd$ with $N_{k,r} = 60 - \tilde{j}$ if $r \equiv 1 \pmod{4}$.

\medskip

\noindent \textbf{Case 2}: If $r \equiv 3 \pmod{4}$, then $\Ftensub$ is a reverse complete Fibonacci subsequence with the 0 followed by a 1 beginning at the index value $\tilde{j} \in \{0, 1, \ldots, 59\}$ of $\Ftensub$. However, in this setting the subsequence moves in reverse through the parent sequence $\Ften$. So the first $\tilde{j}+1$ terms of $\Ftensub$ are exactly the terms $\bigl(\F{\tilde{j}}, \ldots, \F{2}, \F{1}, \F{0} \bigr) = (\F{\tilde{j}}, \ldots, 1, 1, 0)$ of the parent sequence running in reverse. Thus the first term of $\FtensubCompletePeriod$ is the element $\F{\tilde{j}}$. We conclude that $\Ftensub = \FtensubNkrRev$ with $N_{k,r} = \tilde{j}$ if $r \equiv 3 \pmod{4}$.
\end{proof}

\begin{example}
Recall Example~\ref{example:complete_subsequence} where we considered the subsequence $\Ftensub$ with $k=9$ and $r=13$. By Formula~\ref{formula:finding_first_zero}, we have $j_0 \equiv r^{-1}(-k) \equiv 37(-9) \equiv -333 \equiv 12 \pmod{15}$,
and indeed it is readily verified that $\F{9 + 13(12)}$ is the first zero of $\Ftensubfill{9}{13}$. Let us now compute $N_{9,13}$ using the four steps in Algorithm~\ref{algorithm:find_magic_carrot} as follows.
\begin{enumerate}[label=\protect\circled{$\arabic*$}]
    \item Since $r = 13 \equiv 1 \pmod{4}$, we set $\tilde{r} := 3$.
    \item Set $i_0 := \ind_3(\tilde{r}) = \ind_3(3) = 1$.
    \item Set $M_{i_0} := 15 i_0 = 15$. Then $\tilde{j} \equiv r^{-1} (M_{i_0} - k) \equiv 37 (15 - 9)  \equiv 222 \equiv 42 \pmod{60}$,
    and so we set $\tilde{j} := 42$.
    \item Since $r \equiv 1 \pmod{4}$, set $N_{9,13} := 
   60 - \tilde{j}  = 60 - 42 = 18$.
\end{enumerate}
Indeed, by examining the subsequence terms given in Example~\ref{example:complete_subsequence}, it is readily verified that the subsequence $\Ftensubfill{9}{13}$ is exactly the parent sequence $\left(\newF_{10,n}\right)_{n=N_{9,13}}^\infty$ where $N_{9,13} = 18$.
\end{example}

%%%%%%%%%%%%%%%%%%%%%%%%%%%%%%%%%%%%%
%%%%%%%%%%%%%%%%%%%%%%%%%%%%%%%%%%%%%
%%%%%%%%%    SECTION 6    %%%%%%%%%%%
%%%%%%%%%%%%%%%%%%%%%%%%%%%%%%%%%%%%%
%%%%%%%%%%%%%%%%%%%%%%%%%%%%%%%%%%%%%

\section{Open questions}\label{sec:open questions}

In Section~\ref{sec:certain_subsequences}, we made some observations on a variety of specific types of subsequence and we left much for the reader to explore. We add to that in this section but provide more general open questions related to our research.

\begin{question}
What can we generalize to arbitrary moduli $m$? For example, consider these open questions.
\begin{enumerate}
    \item In the $m=10$ setting, we found that $r$ values in the set $U(60)$, namely the $r$ values such that $\gcd(r,60)=1$, the subsequences $\Ftensub$ yielded the complete $\FtensubNkrFwd$ sequence either forward or reverse. Noting that the Pisano period $\pi(10)$ equals 60, we question whether for arbitrary moduli values $m$, will $r$ values in the set $U(\pi(m))$ give subsequences $\Fmnsub$ yielding the complete $\left(\newF_{m,n}\right)_{n = N_{m,k,r}}^\infty$ sequence, either forward or reverse, for some constant $N_{m,k,r} \in \{ 0, 1, \ldots, \pi(m)-1 \}$ dependent on $m$, $k$, and $r$? That is, if $\gcd(r, \pi(m)) = 1$, then do we get a Type~3 subsequence generalization for $\Fmnsub$?
    \item Do there exist moduli values $m$ for which there is no Type 3 subsequence generalization for $\Fmnsub$? That is, do there exist $m$ values such that the subsequence $\Fmnsub$ never recovers the original sequence $\Fmn$ for any jump size $r$?
    \item In the $m=10$ setting in Theorems~\ref{thm:forward_quasi} and \ref{thm:reverse_quasi}, we gave sufficiency conditions on the value $r$ that would guarantee $\Ftensub$ is either a forward or reverse quasi-Fibonacci subsequence. In particular, these conditions depended on the residue class modulo 4 of odd $r$ values that are not divisible by 3. For arbitrary $m$, can one generalize these sufficiency conditions for the subsequence $\Fmnsub$ of $\Fmn$?
\end{enumerate}
\end{question}

\begin{question}
In Remark~\ref{rem:r_equals_25_gives_Lucas_sequence}, we observed that $\Ftensubfill{3}{25}$ and $\Ftensubfill{3}{5}$ both yield the Lucas sequence modulo 10. Can we classify all $k$ and $r$ values such that $\Ftensub$ is the Lucas sequence modulo 10?
\end{question}

\begin{question}
In Lemma~\ref{lem:tantalizing_congruence_result}, we stated that for $r$ values relatively prime to 60, we have the following two tantalizing observations:
\begin{itemize}
    \item If $r \in U(60)$ and $r \equiv 1 \pmod{4}$, then $\F{r} \equiv r \pmod{10}$.
    \item If $r \in U(60)$ and $r \equiv 3 \pmod{4}$, then $\F{r} \equiv -r \pmod{10}$.
\end{itemize}
This is true by brute force confirmation of the sixteen elements in $U(60)$. However that is a very unsatisfying ``proof'' and merely confirms that the result is true but does not explain \textit{why} it is true. Can one come up with a more illuminating proof of this very beautiful result? Moreover, is there some natural generalization of this result for arbitrary moduli $m$?
\end{question}

\section*{Acknowledgments}
We acknowledge David Cochrane for introducing this topic to the second author Mbirika in Cochrane's YouTube video connecting certain patterns in the Fibonacci sequence modulo 10 to applications in astrology~\cite{Cochrane2015}. We also thank Emily Gullerud for her \TeX\, skills using TiKz to create many of the figures in this paper.

%%%%%%%%%%%%%%%%%%%%%%%%%%%%%%%%%%%%%
%%%%%%%%%%%%%%%%%%%%%%%%%%%%%%%%%%%%%
%%%%%%%%%   BIBLIOGRAPHY    %%%%%%%%%
%%%%%%%%%%%%%%%%%%%%%%%%%%%%%%%%%%%%%
%%%%%%%%%%%%%%%%%%%%%%%%%%%%%%%%%%%%%

%%%%%%%%%%%%%%%%%%%%%%%%%%%%%%%%%%%%%
%%%%%%%%%%%%%%%%%%%%%%%%%%%%%%%%%%%%%
%%%%%%%%%     APPENDIX      %%%%%%%%%
%%%%%%%%%%%%%%%%%%%%%%%%%%%%%%%%%%%%%
%%%%%%%%%%%%%%%%%%%%%%%%%%%%%%%%%%%%%

\begin{appendices}
\renewcommand{\thesection}{A}
\section*{Appendix}
%\tocless %%% Comment out this line to list appendix subsections in the Table of Contents
\subsection{Step-by-step construction of \texorpdfstring{$\Ftensub$}{the subsequence} when \texorpdfstring{$k=3$}{k=3} and \texorpdfstring{$r=25$}{r=25}}\label{appendix_A}

\begin{center}
{\renewcommand{\arraystretch}{1.2}
\begin{tabular}{|c|c|}
\hline
$\left(\Ftenfill{3}{25}\right)_{n=0}^1 = (2,1)$ & $\left(\Ftenfill{3}{25}\right)_{n=0}^2 = (2,1,3)$ \\ \hline
\modcirclenk{25}{3}{3}{0}{0}{.5} & \modcirclenk{25}{3}{3}{0}{1}{.5} \\ \hline\hline
$\left(\Ftenfill{3}{25}\right)_{n=0}^{3} = (2,1,3,4)$ & $\left(\Ftenfill{30}{25}\right)_{n=0}^{4} = (2,1,3,4,7)$ \\ \hline
\modcirclenk{25}{3}{3}{0}{2}{.5} & \modcirclenk{25}{3}{3}{0}{3}{.5} \\ \hline\hline
$\left(\Ftenfill{3}{25}\right)_{n=0}^{5} = (2,1,3,4,7,1)$ & $\left(\Ftenfill{3}{25}\right)_{n=0}^{6} = (2,1,3,4,7,1,8)$ \\ \hline
\modcirclenk{25}{3}{3}{0}{4}{.5} & \modcirclenk{25}{3}{3}{0}{5}{.5} \\
\hline
\end{tabular}
}
\end{center}

\begin{center}
{\renewcommand{\arraystretch}{1.2}
\begin{tabular}{|c|c|}
\hline
$\left(\Ftenfill{3}{25}\right)_{n=0}^{7} = (2,1,3,4,7,1,8,9)$ & $\left(\Ftenfill{3}{25}\right)_{n=0}^{8} = (2,1,3,4,7,1,8,9,7)$ \\ \hline
\modcirclenk{25}{3}{3}{0}{6}{.5} & \modcirclenk{25}{3}{3}{0}{7}{.5} \\ \hline\hline
$\left(\Ftenfill{3}{25}\right)_{n=0}^{9} = (2,1,3,4,7,1,8,9,7,6)$ & $\left(\Ftenfill{3}{25}\right)_{n=0}^{10} = (2,1,3,4,7,1,8,9,7,6,3)$ \\ \hline
\modcirclenk{25}{3}{3}{0}{8}{.5} & \modcirclenk{25}{3}{3}{0}{9}{.5} \\ \hline\hline
$\left(\Ftenfill{3}{25}\right)_{n=0}^{11} = (\text{\small 2,1,3,4,7,1,8,9,7,6,3,9})$ & $\left(\Ftenfill{3}{25}\right)_{n=0}^{12} = (\text{\small 2,1,3,4,7,1,8,9,7,6,3,9,2})$ \\ \hline
\modcirclenk{25}{3}{3}{0}{10}{.5} & \modcirclenk{25}{3}{3}{0}{11}{.5} \\
\hline
\end{tabular}
}
\end{center}

\newpage

\subsection{The research team}
The research team for this project are Miko Scott, Dr.~aBa Mbirika, and Dan Guyer (shown left-to-right in the figures below). We acknowledge a number of \textit{research centers} in Eau Claire, Wisconsin, where we conducted much of this research, and we thank them for their hospitality.

\begin{figure}[H]
\begin{center}
    \includegraphics[width=4in]{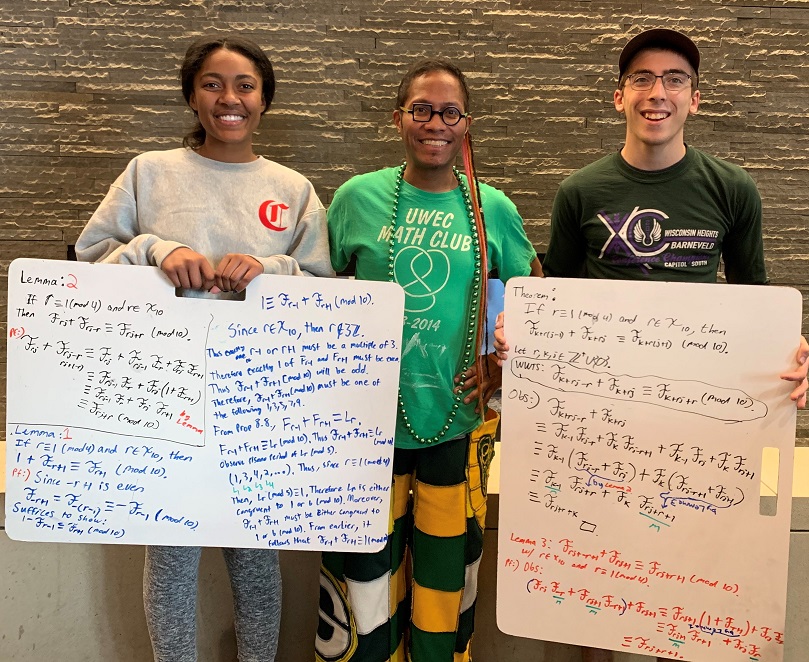}
\end{center}
%\vspace{-.25in}
\caption{The research team working on some Section~\ref{sec:complete_fibonacci_subsequences} results at ECDC coffee shop}
\end{figure}

\begin{figure}[H]
\begin{center}
    \includegraphics[width=4in]{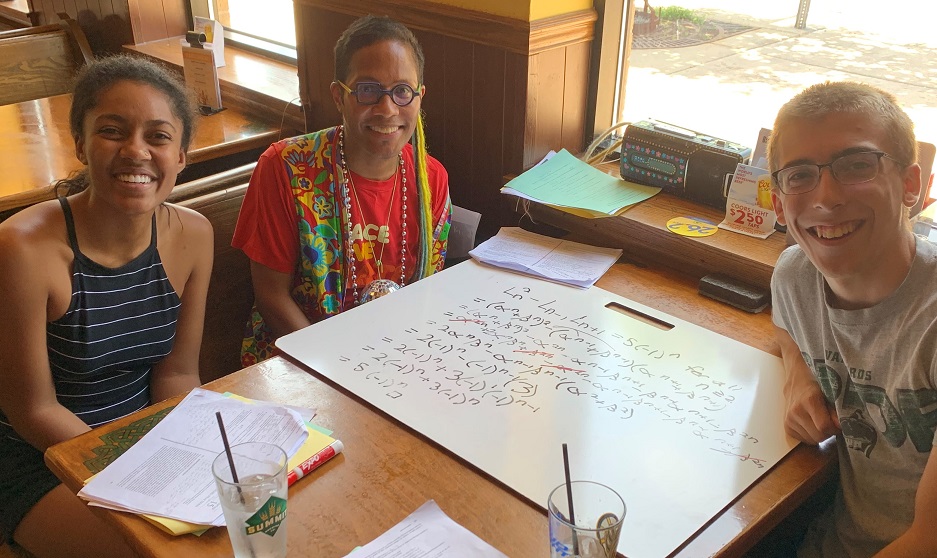}
\end{center}
%\vspace{-.25in}
\caption{The research team working on a Lucas sequence identity at Dooley's Irish Pub}
\end{figure}

\end{appendices}

\end{document}